\documentclass[leqno,11pt, letterpaper]{amsart}
\usepackage{graphicx}
\usepackage{amscd}
\usepackage{amsmath}
\usepackage{caption}
\usepackage{amsfonts}
\usepackage{amssymb}
\usepackage{mathrsfs}
\usepackage{multicol}
\usepackage{color}

\numberwithin{equation}{section}

\definecolor{darkgreen}{rgb}{0.09, 0.45, 0.27}
\definecolor{debianred}{rgb}{0.84, 0.04, 0.33}
\definecolor{orange}{rgb}{1.0, 0.5, 0.0}
\newcommand{\R}{\mathbb R}
\newcommand{\CC}{\mathbb C}
\newcommand{\be}{\begin{equation}}
\newcommand{\ee}{\end{equation}}
\newcommand{\ben}{\begin{eqnarray*}}
\newcommand{\een}{\end{eqnarray*}}

\newtheorem{theorem}{Theorem}
\newtheorem{lemma}{Lemma}
\newtheorem{remark}{Remark}

\newtheorem{proposition}{Proposition}
\newtheorem{definition}{Definition}
\allowdisplaybreaks

\allowdisplaybreaks[1]

\usepackage{color}
\usepackage[normalem]{ulem}
\definecolor{DarkBlue}{rgb}{0,0.1,0.7}  
\definecolor{DarkGreen}{rgb}{0,0.5,0.1}

\newcommand\soutD{\bgroup\markoverwith
	{\textcolor{DarkGreen}{\rule[.5ex]{2pt}{1pt}}}\ULon}
\newcommand{\Hm}[1]{\leavevmode{\marginpar{\tiny%
			$\hbox to 0mm{\hspace*{-0.5mm}$\leftarrow$\hss}%
			\vcenter{\vrule depth 0.1mm height 0.1mm 
		    width \the\marginparwidth}%
			\hbox to
			-0.5mm{\hss$\rightarrow$\hspace*{-0.5mm}}$\\
			\relax\raggedright #1}}}

\begin{document}


\title[Quasilinear electromagnetic Schr\"odinger equations]
{Critical quasilinear Schr\"odinger equations with electromagnetic fields}

\author[L. Baldelli]{Laura Baldelli} \email{labaldelli@ugr.es}

\author[R. Filippucci]{Roberta Filippucci} \email{roberta.filippucci@unipg.it}

\author[D.  Krej\v ci\v r\'ik]{David  Krej\v ci\v r\'ik}
\email{david.krejcirik@fjfi.cvut.cz}

\address[Filippucci]{Department of Mathematics -- University of Perugia --
Via Vanvitelli 1 -- 06123 Perugia,  Italy}
\address[Baldelli]{IMAG, Departamento de An\'alisis Matem\'atico -- Universidad de Granada -- Campus Fuentenueva -- 18071 Granada, Spain}
\address[ Krej\v ci\v r\'ik]{Department of Mathematics, Faculty of Nuclear Sciences and Physical Engineering -- Czech
Technical University in Prague -- Trojanova 13 -- 120 00 Prague, Czech Republic}


\begin{abstract}
The $p$-Laplace operator in the entire $N$-dimensional Euclidean space, subject to external electromagnetic potentials, is investigated. In the general case $1<p<N$, the existence of at least one solution of mountain pass type to a weighted critical equation is proved. Our technique relies on variational methods and faces a twofold difficulty: double lack of compactness, which requires concentration compactness arguments; and a complex quasilinear framework, which entails appropriate inequalities.
\end{abstract}


\maketitle

\section{Introduction}
 
Schr\"odinger operators with electric and magnetic fields
are linear differential operators based on the Laplacian, 
subject to zero- and first-order perturbations, respectively. 
They have been intensively analysed during the last century 
in the context of quantum mechanics, where they model the total
energy of a nonrelativistic particle interacting with external
electromagnetic fields.

The study of magnetic effects is more subtle.
This is primarily because perturbations by first-order differential operators
are more complicated than multiplications by scalar functions.
What is more, the magnetic field enters 
the quantum system indirectly through its vector potential.  
Finally, the magnetic field represents an imaginary perturbation
to the Laplacian, which leads to the necessity of working 
with complex-valued functional spaces and to the lack of 
traditional tools like the maximum principle.

These magnetic intricacies are probably behind the fact
that the nonlinear generalisations
based on the $p$-Laplacian were not considered in the literature 
until very recently~\cite{ckl}.   
In that paper, the framework of the magnetic $p$-Laplacian
is introduced and Hardy-type inequalities are derived.
See also the follow-up~\cite{Chen-Tang} for some improvements 
 as well as the survey \cite{npsv} where the authors extended to the magnetic setting several characterisations of Sobolev and BV functions.

The objective of the present paper is to tackle 
the unprecedented betrothal of the nonlinear $p$-Laplacian
with the magnetic field in yet another unexplored context:
we are concerned with the existence of solutions
to the semilinear stationary electromagnetic Schr\"odinger equation 
with critical growth in~$\R^N$ with $N > 1$. 
To be more specific, given an electric and magnetic potentials 
$V:\R^N \to \R$ and $A:\R^N \to \R^N$, respectively, we consider the equation
  
\begin{equation}\label{prob}
-\Delta_{A,p} u+V(x)|u|^{p-2}u=f(x,|u|^p)|u|^{p-2}u+K(x)|u|^{p^*-2}u, \qquad x\in\mathbb R^N.
\end{equation}
 
Here $\Delta_{A,p} u =  \nabla_A \cdot (|\nabla_A u|^{p-2}\nabla_Au)$ 
is the magnetic $p$-Laplacian with $1<p<N$ 
and $\nabla_A = \nabla+iA$ is the magnetic gradient, 
$p^*=Np/(N-p)$ is the Sobolev's critical exponent
and the functions $f:\R^N\times\R\to\R$ and $K:\R^N\to\R$
model subcritical and critical nonlinearities, respectively.
Our goal is to identify suitable assumptions on the functions 
$V,f,K$ and notably~$A$ which guarantee the existence of solutions~$u$ to \eqref{prob}.

Despite the vector potential~$A$ appears explicitly in~\eqref{prob}, 
the quantity with more physical relevance is the magnetic field $B = dA$
satisfying the Maxwell equation $dB = 0$,
where $d$ is the exterior derivative.
In particular, given~$B$, the potential~$A$ is not uniquely determined.
This is called gauge invariance in physics.
In our context, choosing $\tilde A = A + d\phi$ with $\phi:\mathbb R^N \to \mathbb R$
(so that $d\tilde A = dA$), the respective solutions $\tilde u,u$ of~\eqref{prob} 
merely differ by a phase factor, namely $\tilde u = e^{-i\phi} u$.

Note that equation~\eqref{prob} involves three types of nonlinearities given by the quasilinear nature of the operator ($p \not= 2$), the critical ($K \not= 0$) and the subcritical ($f\not=0$) terms in the reaction.
The traditional nonlinear problem is the semilinear electromagnetic 
Schr\"odinger equation with $p=2$ and $K = 0$ .
This subcritical case was investigated in the pioneering papers~\cite{el} 
by minimisation arguments for $N=2,3$, in \cite{kur} via variational methods, 
see \cite{st} for $N=3$. 
In the critical case (still $p=2$ but $K$~nonzero),  
the magnetic Laplacian was studied by min-max arguments by~\cite{AS3} and~\cite{cs5}, 
where the existence of at least a nontrivial solution is established. 

There are also many works which consider particular realisations of~\eqref{prob}  
when a small semiclassical parameter $\hbar \to 0^+$ is added to the gradient term, 
namely the replacement $\nabla \mapsto \hbar\nabla$ is considered.
We refer to the pioneering works \cite{cs,p} 
taking into account different type of nonlinearities. 
More recently, Ji and R\u{a}dulescu \cite{JR21, JR21ad} 
and Ambrosio \cite{aAS, ajmaa} consider the magnetic situation (still with $p=2$) 
and prove multiplicity  and
concentration of nontrivial solutions for $\varepsilon>0$ small by Lusternik--Schni\-rel\-man theory 
and the penalisation technique. 
We also mention \cite{zj13, fzs} where there is an attempt to deal 
with the semiclassical setting, by using a different definition of the magnetic $p$-Laplacian.

The paper of Brezis and Nirenberg~\cite{BN} paved the way in the study of critical nonlinear Schr\"odinger equation \eqref{prob} without the magnetic field (i.e.\ $A=0$).
Later, existence and multiplicity results for critical problems driven by general quasilinear operators were gained, see \cite{GPeral,  SY, AP, BF, BFdeg, fpr} and the references therein. 

The given bibliography is certainly far to be complete because of the huge number of papers devoted to quasilinear critical problems.
As far as we know, however, equation~\eqref{prob} with simultaneously 
$p\not=2$ and $A \not= 0$ has not been studied so far.  The objective of this paper is to fill in this gap. Specifically, 
the novelty of this paper consists in undertaking the study of the magnetic nonlinear Laplacian from a variational perspective, thereby establishing an existence result, as far as we are aware, for the first time in the existent literature. Despite the fact that the method's foundation is rooted in the Mountain Pass Theorem, the nonlinear nature of the magnetic context renders the procedure considerably more intricate and delicate in comparison to the linear case, which has been extensively explored in the current literature.
Motivated by these papers, we focus our attention 
on the existence of nontrivial solutions.
To the best of our knowledge, this is the first time that non-concentrating solutions to the magnetic $p$-Laplacian are taken into account.

To state our main result, let us collect characteristic hypotheses
about $V,f,K$ and~$A$.

\begin{itemize} 
 \item[$(f_0)$] $f\in C(\mathbb R^N, \mathbb R_0^+)$ 
 and $\lim_{t\to0}f(x,t)=0$ uniformly in $x\in\mathbb R^N$; 
\item[$(f_1)$]  there exist two nonnegative weights $h_1\in C(\mathbb R^N)\cap L^{N/p}(\mathbb R^N)\cap L^{p^*/(p^*-k)}(\mathbb R^N)$, where $N/p=p^*/(p^*-p)$,  $ 0<h_2\in C(\mathbb R^N)\cap L^{p^*/(p^*-k)}(\mathbb R^N)$ with $p<k<p^*$ such that
$$f(x, t)\le h_1(x)+h_2(x)t^{(k-p)/p}$$
for all $(x, t)\in \mathbb R^N\times \mathbb R_0^+$;
\item[$(f_2)$]  there exists $p<\theta<p^*$ such that for all $(x, t)\in \mathbb R^N\times \mathbb R^+$ it holds
$$0<\frac{\theta}{p}F(x, t) \le f (x, t)t,\qquad F(x,t)=\int_0^{t}f(x,s)ds \,;$$
\item[$(f_3)$] there exist $\lambda\in\R$  and $q\in(p, k)$ such that $F(x, t)\ge \lambda t^{q/p}$ for all $t\ge0$, where 
\begin{itemize}
\item[$\circ$] $\lambda>0$ if $N\ge p^2$ and $1<p\le 2$,
\item[$\circ$]  $\lambda$ sufficiently large if $N<p^2$ and $1<p\le 2$,
\item[$\circ$] $\lambda$ sufficiently large if $p>2$;
\end{itemize}
\item[$(A)$] $A\in C(\mathbb R^N, \mathbb R^N)$;
\item[$(K)$] $0\le K\in C(\mathbb R^N)\cap L^\infty(\mathbb R^N)$ with $\|K\|_\infty=K(0)>0$ and
$$K(x)=K(0)+O(|x|^\tau)$$
in $B(0,\delta_K)$, for some $\delta_K>0$, where
$$\frac{N}{p-1}>\tau>
\begin{cases}
p &\quad N> p^2,\\
\frac{N-p}{p-1} &\quad N\le p^2;
\end{cases}
$$
\item[$(V)$]
$V\in C(\mathbb R^N)$ with ${\displaystyle \inf_{x\in\mathbb R^N} V(x)}=V_0>0$.
\end{itemize}

Actually, condition $(f_2)$ forces that $s \mapsto s^{\theta/p}/F(x,s)$ is nonincreasing in the entire $\mathbb R^+$, 
from which $F(x,s)\ge C s^{\theta/p}$,  $s\ge s_0>0$
with $C= F(x,s_0)s_0^{-\theta/p}$. Thus, $(f_2)$ and $(f_3)$ are comparable in some cases.
Furthermore, condition $(f_1)$ for $t=0$ is trivially satisfied as a consequence of $(f_0)$. 
Moreover, as a consequence of $(f_0)$ and $(f_1)$, for all $\xi>0$
there exists $C_\xi>0$ such that
\begin{equation}\label{f0f1}
f(x,t)\le \xi + C_\xi h_3(x)t^{(k-p)/p}, \qquad (x,t)\in\mathbb R^N\times\mathbb R_0^+,
\end{equation}
where $h_3(x)=h_1(x)+h_2(x)\in L^{p^*/(p^*-k)}(\mathbb R^N)$ 
by summability of $h_1, h_2$ in $(f_1)$. Note that a similar property holds also for $F$, see the first lines of the proof of Lemma \ref{mpt_geometry}.
In order to define the operator associated with
the magnetic $p$-Laplacian $-\Delta_{A,p}$ variationally,
condition $(A)$ can be relaxed to $A \in L_\mathrm{loc}^{p}(\R^N,\R^N)$. 
However, for technical reasons (see the proof of Lemma~\ref{c<csegnato}),
we have decided to assume its continuity.
Finally, note that the flatness condition $(K)$ has been imposed by several authors as \cite{SY, dh}.

Our main result is the following.

\begin{theorem}\label{main}
Let $(f_0), (f_1), (f_2), (f_3), (A), (K), (V)$ hold. 
Then there exists at least one nontrivial weak solution of problem \eqref{prob}.
\end{theorem}

The proof of Theorem \ref{main} is obtained using suitable variational mountain pass arguments. 
We emphasise that when dealing with the magnetic case, 
it is rather difficult to obtain compactness by overcoming the double lack of compactness,
due to the critical growth of the nonlinearity and the presence of the entire $\mathbb R^N$, because of the necessary complex nature of the functional spaces. Indeed, compared with the subcritical case, a more careful analysis should be performed
and many classical tools must be reformulated spending much more attention. 
More specifically, after proving the mountain pass geometry, we use the concentration-compactness principle of Lions, to verify that the Palais--Smale condition is regained below a suitable level related to the best constant of the Sobolev embedding and the weight associated with the critical term in the nonlinearity. We also underline that in the final part of the proof we use a special inequality which is known, as far as we see, only in the real case. Moreover, in proving that the mountain pass level is under the obtained threshold, we perform a very delicate analysis of the functional behaviour that also includes the application of the Talenti function.
 
The paper is structured as follows. 
In Section~\ref{prel}, we collect some classical definitions 
as well as
some useful results, including the Mountain Pass Lemma.
In Section~\ref{pssec} we prove first some properties of Palais--Smale sequences and then,  we check the validity for the functional of the Mountain pass geometry. 
As a consequence, Theorem~\ref{main} is established at the end of this section. 
Finally, Appendix~\ref{ineq} contains the proof of a useful complex inequality.

\section{Preliminaries}\label{prel}

We start with setting some basic notations. We indicate with $B_r(x)$ the open $\R^N$-ball of centre $x\in\R^N$ and radius $r>0$, omitting $x$ when it is the origin.
Let $M(\R^N,\R)$ be the space of all finite signed Radon measures. Concerning convergence of sequence of measures $(\mu_n)_n\subseteq M(\R^N,\R)$, we write $\mu_n\stackrel{*}{\rightharpoonup}\mu$ and $\mu_n\rightharpoonup\mu$ to signify tight and weak convergence, respectively (see \cite{BF} for more details).

For any complex number $a\in\CC^N$  we use the notation $a=a_R+ia_I$ with $a_R,a_I\in \R^N$, or equivalently $a=\Re(a)+i\Im(a)$, according to the more suitable. 
Letting $N\ge 1$, the scalar product in $\R^N$ between $u,v\in\R^N$ is denoted by $u\cdot v\in\R$, while the complex product in $\CC^N$ between $u,v\in\CC^N$ is denoted by $\langle u,v\rangle_{\CC^N}=u \cdot \overline{v}$, see Appendix \ref{ineq}. For the case $N=1$ we have $\langle u,v\rangle_{\CC}=u \cdot \overline{v}=u_Rv_R+u_Iv_I+i(u_Iv_R-u_Rv_I)\in\mathbb C$ for any $u,v\in\CC$.

Given any measurable set $\Omega\subseteq \R^N$ and $q\in[1,+\infty]$, $L^q(\Omega)$ stands for the standard Lebesgue space, whose norm will be indicated with $\|\cdot\|_{L^q(\Omega)}$, or simply $\|\cdot\|_q$ when $\Omega=\R^N$. 

For $p\in(1,N)$, we will also make use of the Beppo Levi space $D^{1,p}(\R^N)$, which is the closure of $C^\infty_c(\R^N)$ with respect to the norm
$\|u\|_{D^{1,p}(\R^N)}=\|\nabla u\|_p$.
Sobolev's theo\-rem ensures that $D^{1,p}(\R^N) \hookrightarrow L^{p^*}(\R^N)$ continuously; the best (embedding) constant $c$ in the Sobolev inequality $\|u\|_{L^{p^*}(\R^N)}\leq c \|u\|_{D^{1,p}(\R^N)}$ is $S^{-1/p}$, being
\begin{equation}\label{S}
S= \inf_{u\in D^{1,p}(\R^N)\setminus\{0\}} \frac{\|\nabla u\|_p^p}{\|u\|_{p^*}^p}. 
\end{equation}
It is known explicitly (see, e.g., \cite[Thm.~8.3]{LL}).
According to Sobolev's theorem, one has $D^{1,p}(\R^N) = \left\{u\in L^{p^*}(\R^N): \, |\nabla u|\in L^p(\R^N)\right\}.$

In what follows,  $C$ will denote a positive constant which may change its value at each passage.

Following \cite{ck,ckl,npsv}, the magnetic Sobolev space $W^{1,p}_A(\R^N, \mathbb C)$ is given by
$$W^{1,p}_A(\R^N, \mathbb C)=\left\{u\in L^p(\R^N,\CC) : h_{A,p}[u]<\infty\right\},$$
where
$$h_{A,p}[u] =\left(\int_{\mathbb R^N} |\nabla_A u|^p dx\right)^{1/p} =\left( \int_{\mathbb R^N} |\nabla u+iA(x)u|^p dx\right)^{1/p}.$$
The space $W^{1,p}_A(\R^N, \mathbb C)$ is equipped with the norm
$$\|u \|_{A,p}=\biggl((h_{A,p}[u])^p + \int_{\mathbb R^N}|u|^p dx\biggr)^{1/p}.$$

In order to deal with problem \eqref{prob} where a scalar potential term~$V$ appears, 
we introduce the electromagnetic Sobolev space $W^{1,p}_{A, V}(\mathbb R^N, \mathbb C)$ given by
$$W^{1,p}_{A, V}(\mathbb R^N, \mathbb C)=\overline{C_0^\infty(\mathbb R^N, \mathbb C)}^{\|\cdot \|_{A,p,V}}$$
where
$$\|u\|_{A,p,V}=\biggl(\int_{\mathbb R^N} |\nabla_A u|^p+V(x)|u|^p dx\biggr)^{1/p}.$$
Note that the closure is well defined because~$V$ is positive due to hypothesis~$(V)$.  
For simplicity we call $X=W^{1,p}_{A, V}(\mathbb R^N, \mathbb C)$ 
and $\|\cdot\|=\|\cdot\|_{A,p,V}$.

Given any $u\in X$, the functional associated to \eqref{prob} reads
 
\begin{equation}\label{funct}
J_A(u)=\frac1p\|u\|^p-\frac1{p^*}\int_{\mathbb R^N}K(x)|u|^{p^*}dx-\frac{1}{p}\int_{\mathbb R^N}F(x,|u|^p)dx,
\end{equation}
 
where
$$F(x,t^p)=\int_0^{t^p}f(x,s)ds=p\int_0^{t} f(x,s^p)s^{p-1}\, ds
, \qquad t>0.$$

An important tool in the study of magnetic fields is the diamagnetic inequality 
also called Kato's inequality (see, e.g., \cite[Thm.~5.3.1]{bel}) 
 
\begin{equation}\label{kato}
|\nabla_A u(x)| \ge  |\nabla|u(x)||
\quad \text{ for a.e.}\quad  x \in \mathbb R^N,
\end{equation}

valid for every $u \in X$.
 In particular, denoting $
  L^p_V(\mathbb R^N, \mathbb C)
  = \{u \in L^p(\R^N,\mathbb{C}) :  \int_{\mathbb R^N} V(x)|u|^p dx<\infty \}
$ the Lebesgue weighted space, by \eqref{kato}, it is clear that $X\subset W^{1,p}_{V}(\mathbb R^N, \mathbb C)$, where
$W^{1,p}_{V}(\mathbb R^N, \mathbb C)=W^{1,p}(\mathbb R^N, \mathbb C)\cap L^p_V(\mathbb R^N, \mathbb C)$ endowed with the norm 
$$\|u\|_{p,V}=\biggl(\int_{\mathbb R^N} |\nabla  u|^p+V(x)|u|^p dx\biggr)^{1/p}.$$
Actually, it holds $X=W^{1,p}_{V}(\mathbb R^N, \mathbb C)$ if $A$ is bounded, as observed in \cite{ckl}.

Furthermore, thanks to $(V)$, if $u\in X$ then  $u\in L^s(\mathbb R^N, \mathbb C)$  
for any $s\in[p,p^*]$, namely the embedding
$$X\hookrightarrow L^s(\R^N, \CC)$$
is continuous for any $s\in [p,p^*]$. Thus, for each $s\in[p,p^*]$ there exists $C_s>0$ such that 
\begin{equation}\label{green}
\|u\|_s\le C_s\biggl(\int_{\mathbb R^N} |\nabla|u||^p\biggr)^{1/p}\le C_s \|u\|.
\end{equation}
Moreover the embedding
\begin{equation}\label{compem}
X\hookrightarrow \hookrightarrow L^s_\mathrm{loc}(\R^N, \CC)
\end{equation}
is compact for any $s\in [1,p^*)$, see \cite{AP}.

In order to study \eqref{prob}, we observe that
$$\nabla_A u=(D^1_Au, \dots, D^N_Au), \qquad 
D^j_Au=\partial_ju+iA(x)u,\, \,\, j=1,\dots,N $$
and, for any $\eta\in C^1(\mathbb R^N, \mathbb R)$ it holds the formula
\begin{equation}\label{for1}
\overline{\nabla_A(u\eta)}=\,\overline{u} \,\nabla\eta +\eta \overline{\nabla_Au},\end{equation}
which is, since $\overline{\nabla_A u}=\bigl(\overline{D^1_Au}, \dots, \overline{D^N_Au}\bigr)$, equivalent to
$\overline{D_A^j(u\eta)}=\,\overline{u}\,\partial_j \eta+\eta\overline{D_A^j u}.$

A classical strategy to prove existence results for problem \eqref{prob} consists in looking for critical points of the $C^1$ functional defined in \eqref{funct}.
The differential of $J_A$ can be written as
$$\langle J'_A(u),v\rangle=r_{A,p}(u, v)
-\Re \biggl(\int_{\mathbb R^N}K(x)|u|^{p^*-2}\langle u, v \rangle_{\CC^N}dx+\int_{\mathbb R^N} f(x,|u|^p)|u|^{p-2}\langle u, v \rangle_{\CC^N}  dx\biggr),$$
where  $u,v\in X$ and 
$$r_{A,p}(u, v)=\Re \biggl(\int_{\mathbb R^N} |\nabla_A u|^{p-2}\langle \nabla_A u, \nabla_A v \rangle_{\CC^N}  dx+\int_{\mathbb R^N} V(x)|u|^{p-2}\langle u, v \rangle_{\CC^N}  dx\biggr).
$$

Since we deal with critical problems in the entire Euclidean space, 
we necessarily face the so-called double lack of compactness. 
It is customary to recover compactness by using 
the concentration compactness principle around finite points 
by Lions~\cite{LionsRev1}, but under the requirement of tight convergence when $\Omega$ is unbounded,
 or equivalently the concentration compactness at infinity by Ben-Naoum et al.~\cite{BTW}, aimed
to avoid also concentration of mass at infinity, which roughly means tightness. 
For completeness, we recall these two methods here.

\begin{proposition}[{\cite[Lem.~I.1]{LionsRev1}}]\label{lem4.1} 
Let $\Omega\subset \mathbb R^N$ be a domain and $1\le p< N$.
Let $(v_{n})_{n}$ be a bounded sequence in $D^{1,p}(\Omega,\mathbb R)$ converging weakly to some $v$ and such that
$|\nabla v_{n}|^{p}\rightharpoonup\mu$ and either $|v_{n}|^{p^*}\rightharpoonup\nu$ if $\Omega$ is bounded
or $|v_{n}|^{p^*}\overset{\ast}{\rightharpoonup}\nu$ if $\Omega$ is unbounded, where
$\mu$, $\nu$ are bounded nonnegative measures on $\Omega$. 
Then there exists some at most countable set
$J$ such that
\begin{enumerate}
\item[(i)] $\nu=|v|^{p^*}+ \sum_{j\in J}\nu_{j}\delta_{x_{j}}, \quad \nu_j\ge 0,$
\item[(ii)] $\mu\ge|\nabla v|^{p}+\sum_{j\in J}\mu_{j}\delta_{x_{j}}, \quad \mu_j\ge 0,$
\item[(iii)] $S\nu_{j}^{p/p^*}\le \mu_{j},\qquad \displaystyle{\sum_{j\in J}\nu_{j}^{p/p^*}<\infty},$
\end{enumerate}
where $(x_{j})_{j\in J}$ are distinct points in $\Omega$, $\delta_{x}$ is the Dirac-mass of mass 1 concentrated at
$x\in\Omega$ and $S$ is the Sobolev's constant defined in \eqref{S}.
\end{proposition}

\begin{proposition}[{\cite[Prop.~3.3]{BTW}}]\label{lemben} 
Let $(v_n)_n$ be a bounded sequence in $D^{1,p}(\mathbb R^N,\mathbb R)$ and define
\begin{equation}\label{munuinf}
\nu_\infty=\lim_{R\to\infty} \limsup_{n\to\infty} \int_{|x|>R} |v_n|^{p^*} dx,\qquad
\mu_\infty=\lim_{R\to\infty} \limsup_{n\to\infty} \int_{|x|>R} |\nabla v_n|^{p} dx.
\end{equation}
Then, the quantities $\nu_\infty$ and $\mu_\infty$ exist and satisfy
\begin{equation}\label{unpstarinf}
\limsup_{n\to\infty} \int_{\mathbb R^N} |v_n|^{p^*} dx=\int_{\mathbb R^N} d\nu +\nu_\infty,
\end{equation}
\begin{equation}\label{dunpinf}
\limsup_{n\to\infty} \int_{\mathbb R^N} |\nabla v_n|^{p} dx=\int_{\mathbb R^N} d\mu +\mu_\infty,
\end{equation}
$$S\nu_{\infty}^{p/p^*}\le \mu_\infty,$$
where $\nu$ and $\mu$ are as in {\rm(i)} and {\rm(ii)} in Lemma \ref{lem4.1} and such that {\rm(iii)} holds.
\end{proposition}

Let us recall the following version of the Mountain Pass Lemma.

\begin{theorem}[\cite{AR}]\label{mpthm}
Let $(Y,\|\cdot\|_Y)$ be a Banach space and consider $E\in C^1(Y)$. We assume that
\par (i) $E(0)=0$,
\par (ii) there exist $\alpha,R>0$ such that $E(u)\geq\alpha$ for all $u\in Y$, with $\|u\|_Y=R$,
\par (iii) there exists $v_0\in Y$ such that $\limsup_{t\to \infty}E(tv_0)<0$.
\par\noindent
Let $t_0>0$ be such that $\|t_0v_0\|_Y>R$ and $E(t_0v_0)<0$
and let
$$c=\inf_{\gamma\in\Gamma}\,\sup_{t\in [0,1]}E(\gamma(t)),$$
where
$$\Gamma=\{\gamma\in C^0([0,1],Y)\,/\, \gamma(0)=0\hbox{ and }\gamma(1)=t_0v_0\}.$$
Then, there exists a Palais--Smale sequence at level $c$, that is
 a sequence $(u_n)_n\subset Y$ such that
$$\lim_{n\to \infty}E(u_n)=c\quad\hbox{ and }\quad\lim_{n\to \infty}E'(u_n)= 0\quad\hbox{strongly in }Y'.$$
\end{theorem}

We end this section by defining a well-known family of functions and giving some useful estimates that will be needed to overcome difficulties coming from the critical exponent.
For any $\varepsilon>0$ and $x_0\in\R^N$, we consider the following family of instantons (see \cite{Tal})
\begin{equation}\label{scal}
U_{\varepsilon,x_0}(x) =  c_{p,N}\left[\frac{\varepsilon^{\frac{1}{p-1}}}{\varepsilon^{\frac{p}{p-1}}+|x-x_0|^{\frac{p}{p-1}}}\right]^{\frac{N-p}{p}}=\varepsilon^{-\frac{N-p}{p}} U\left(\frac{x-x_0}{\varepsilon}\right),
\end{equation}
where
\begin{equation}\label{cpndef}
U(x)=c_{p,N}\left(1+|x|^{\frac{p}{p-1}}\right)^{-\frac{N-p}{p}}, \qquad c_{p,N}=\left(N^{\frac{1}{p}}\left(\frac{N-p}{p-1}\right)^{\frac{p-1}{p}}\right)^{\frac{N-p}{p}}.
\end{equation}
In particular,
\begin{equation}\label{normep}
\|\nabla U_{\varepsilon, x_0}\|_p=\|\nabla U\|_p, \qquad \|U_{\varepsilon, x_0}\|_{p^*}=\|U\|_{p^*}
\end{equation}
 and it is well known (\cite{Tal}, \cite{sciunzi}, see also \cite{SY}), that $U_{\varepsilon,x_0}$ for all $\varepsilon>0$, are exactly all the regular radial solutions of the
critical $p$-Laplace equation
\begin{equation}\label{cri}
-\Delta_p u=u^{p^*-1},\qquad u > 0 \,\,\text{in}\,\, \R^N,\qquad u\in D^{1,p}(\R^N).
\end{equation}
From \cite{LionsRev1}, it follows that the family of functions given by \eqref{scal} are minimisers
to \eqref{S}, that is, for all $\varepsilon > 0$ and $x_0\in\mathbb R^N$, we have
$$\|\nabla U_{\varepsilon, x_0}\|_p^p=S\|U_{\varepsilon, x_0}\|_{p^*}^{p}.$$
Since $U_{\varepsilon, x_0}$ solves \eqref{cri}, as already mentioned, integration by parts and the definition of $c_{p,N}$ in \eqref{cpndef}, yields $\|\nabla U_{\varepsilon, x_0}\|_p^p=\|U_{\varepsilon, x_0}\|_{p^*}^{p^*}$, implying in view of \eqref{normep} that
\begin{equation}\label{Uepsx0}
S^{N/p}=\|U_{\varepsilon, x_0}\|_{p^*}^{p^*}.
\end{equation}
Let $0<\delta_\psi<\delta_K$, where $\delta_K$ is given in (K), and $\psi\in C^1$ radial such that 
\begin{equation}\label{defcut}
\psi=1\,\, \text{in}\,\, B_{\delta_\psi/2}(x_0),\qquad \psi=0 \,\, \text{in}\,\,B^c_{\delta_\psi}(x_0).
\end{equation}
Define $\tilde w_{\varepsilon, x_0}=\psi U_{\varepsilon, x_0}$ with $U_{\varepsilon, x_0}$ as in \eqref{scal} and 
\begin{equation}\label{defwtil}
w_{\varepsilon, x_0}=\tilde w_{\varepsilon, x_0}/\|\tilde w_{\varepsilon, x_0}\|_{p^*}. 
\end{equation}
As noted in \cite{ap87}, the asymptotic behaviour of $w_{\varepsilon, x_0}$ for $\varepsilon\to 0$ has to be like $U_{\varepsilon, x_0}$ and the corresponding functional must approximate $S$.
Finally, we report some useful estimates taken from \cite[Lem.~A5]{GP94}, 
\cite[pp.~45-46]{Peralcourse}, see also \cite{ap87, dh}.

\begin{lemma}\label{tallemma} Let $w_{\varepsilon, x_0}$ be defined in \eqref{defwtil}, then it holds
$$\|\nabla w_{\varepsilon, x_0}\|_p^p=S+O(\varepsilon^{\frac{N-p}{p-1}})$$
and
\begin{equation}\label{est_norm_q}\|w_{\varepsilon, x_0}\|_q^q=\begin{cases}O(\varepsilon^{N-\frac{N-p}{p}q}) &\text{if}\,\,q>\frac{N(p-1)}{N-p},\\
O(\varepsilon^{\frac{N-p}{p(p-1)}q}|\log\varepsilon|)&\text{if}\,\,q=\frac{N(p-1)}{N-p},\\
O(\varepsilon^{\frac{N-p}{p(p-1)}q}) &\text{if}\,\,q<\frac{N(p-1)}{N-p} .
\end{cases}\end{equation}
If particular, if $q=p$, we have
$$\|w_{\varepsilon, x_0}\|_p^p=\begin{cases}O(\varepsilon^p) &\text{if}\,\,N>p^2,\\
O(\varepsilon^{\frac{N-p}{p-1}}|\log\varepsilon|)&\text{if}\,\,N=p^2,\\
O(\varepsilon^{\frac{N-p}{p-1}}) &\text{if}\,\,N<p^2.
\end{cases}$$
\end{lemma}

\begin{remark}\label{weW1p}
We point out that, as a consequence of Lemma \ref{tallemma}, we have $w_{\varepsilon,0}\in W^{1,p}(\mathbb R^N)$, for any $\varepsilon>0$ and $x_0\in \R^N$.
\end{remark}

For simplicity, from now on, if $x_0=0$, we call $w_{\varepsilon, 0}$, $\tilde w_{\varepsilon, 0}$, $U_{\varepsilon, 0}$ as, respectively, $w_{\varepsilon}$, $\tilde w_\varepsilon$, $U_\varepsilon$.

\section{ Mountain pass framework}\label{pssec}

We start this section by giving the definition of $(PS)_c$ sequence. 

\begin{definition} 
Let $Y$ be a Banach space and $E: Y\to\mathbb{R}$ be a differentiable functional. A sequence
$(u_{n})_{n}\subset Y$ is called a $(PS)_{c}$ sequence for $E$ if $ E(u_{k})\to c$ and $E'(u_{k})\to 0$
as $k\to\infty$.
Moreover, we say that $E$ satisfies the $ (PS)_{c}$ condition if every $(PS)_{c}$ sequence for $E$ has a converging subsequence in $Y$.
\end{definition} 

First, we investigate the boundedness of every Palais--Smale sequence in our case.
\begin{lemma}\label{lembound} 
Let $(u_n)_n\subset X$ be a Palais--Smale sequence for $J_A$, then 
$(u_n)_n$ is bounded in $X$ for any $c\in\mathbb R$.
\end{lemma}
\begin{proof}
Let $(u_{n})_{n}\subset X$ be a $(PS)_{c}$ sequence of $J_A$ for all $c\in\mathbb{R}$
 namely, 
\begin{equation}\label{PS_prop}J_A(u_{n})=c+o(1), \,\ J'_A(u_{n})=o(1)\quad \text{as} \quad n\to\infty,\end{equation}
so that $|\langle J'_A(u_{n}),u_{n}\rangle|\le \|u_{n}\|$ for $n$ large.
In particular, by $(f_2)$, $(K)$ and being $p<\theta<p^*$ we have
$$\begin{aligned}
c +o(1) +o(1)\|u_{n}\|&= J_A(u_{n})- \frac{1}{\theta}\langle J'_A(u_{n}),u_{n}\rangle\\
&= \left(\frac{1}{p}-\frac{1}{\theta}\right)\| u_n\|^p
-\biggl(\frac 1{p^*} -\frac1\theta \biggr)\int_{\mathbb R^N}K(x)|u_n|^{p^*}dx\\&\quad-\frac1p \int_{\mathbb R^N}F(x,|u_n|^p)dx+\frac{1}{\theta} \int_{\mathbb R^N} f(x,|u_n|^p)|u_n|^{p} dx\\
&\ge \left(\frac{1}{p}-\frac{1}{\theta}\right)\| u_n\|^p.
\end{aligned}$$
Thus, it immediately follows that
$\|u_n\|$ should be bounded, being $p>1$.
\end{proof}

Inspired by \cite{aAS} and \cite{dl07} we prove the following.

\begin{lemma}\label{mpt_geometry}
The functional $J_A$ has the Mountain Pass geometry, namely the assumptions of Theorem \ref{mpthm} are satisfied.
\end{lemma}

\begin{proof}
We have to verify the hypotheses $(i)-(iii)$ of the Mountain Pass Lemma.
From Section \ref{prel}, $J_A\in C^1(X)$ and clearly  $(i)$ is satisfied since $J_A(0)=0$.
Now we prove $(ii)$ of Theorem
\ref{mpthm}. 
By l'Hospital rule and $(f_0)$, we have
$$\lim_{u\to0}\frac{F(x,|u|^p)}{|u|^p}=
\lim_{u\to0}\frac{f(x,|u|^p)p|u|^{p-2}u}{p|u|^{p-2}u}=
\lim_{u\to0}f(x,|u|^p)=f(x,0)=0,$$
so that  for all $\delta>0$ then $F(x,|u|^p)\le \delta|u|^p$ for $|u|$ small and any $x\in \R^N$.
On the other hand, by $(f_2)$ and $(f_1)$, being $p<\theta$,
$$F(x,|u|)\le \frac{p}{\theta} |u|f(x, |u|) \le h_1(x)|u|+h_2(x)|u|^{k/p},$$
so that for $|u|$ large, since $k>p$ we have
$F(x,|u|)\le h_3(x) |u|^{k/p}
$,   with $h_3=h_1(x)+h_2(x)>0$ in $\mathbb R^N$ by $(f_1)$.
In turn, for all $\delta>0$, sufficiently small, there exists $c_\delta>0$ 
$$F(x,|u|^p)\le  \delta  |u|^p+c_\delta h_3(x)|u|^k
$$
for all $u\in X$ and $x\in\mathbb R^N$. 
Indeed, it easily follows from the boundedness from above in $\mathbb R^N\times \mathbb R^+$ of the function 
$$g_\delta(x,u):=\biggr[\frac{F(x,|u|^p)}{|u|^p}-\delta\biggr]\frac1{h_3(x)|u|^{k-p}}$$
being 
$$g_\delta(x,u)\sim - \frac\delta {h_3(x)|u|^{k-p}} \mbox{ if } u\to 0^+,\qquad g_\delta(x,u)\le  1 - \frac\delta {h_3(x)|u|^{k-p}} \mbox{ if } |u|\to \infty
$$
yielding $\sup_{\mathbb R^N\times \{|u|\le1\}}g_\delta(x,u)\le0$
and $\sup_{\mathbb R^N\times \{|u|>>1\}}g_\delta(x,u)\le 1$. On the other hand  if $0< A\le |u|\le B<\infty $
$$g_\delta(x,u)\le \frac{h_1(x)}{h_3(x)|u|^{k-p}}+\frac{h_2(x)}{h_3(x)}-\frac\delta{h_3(x)|u|^{k-p}} \le A^{p-k}+1-\delta A^{p-k} \frac 1{h_3(x)},
$$
being $h_1,\, h_2\le h_3$, then $\sup_{\mathbb R^N\times [A,B]}g_\delta(x,u)\le A^{p-k}+1$.

Now we use Sobolev's inequality,  
\eqref{green} and the calculations above to get
$$\begin{aligned}
J_A(u)&= \frac1p\|u\|^p-\frac1{p^*}\int_{\mathbb R^N}K(x)|u|^{p^*}dx-\frac{1}{p}\int_{\mathbb R^N}F(x,|u|^p)dx,
\\&\ge \frac1p\|u\|^p-\frac{\|K\|_{\infty}}{p^*C_{p^*}^{p^*}}\|u\|^{p^*}-\delta\|u\|_p^p- c\|h_3\|_{p^*/(p^*-k)}\|u\|_k^k\\
&\ge \|u\|^p\left[C_1- C_2\|u\|^{p^*-p}-C_3\|u\|^{k-p}\right],
\end{aligned}$$
with
$$C_1=\frac 1p-C^p\delta, \qquad C_2=\frac{\|K\|_{\infty}}{p^*C_{p^*}^{p^*}}, \qquad C_3=cC^k\|h_3\|_{p^*/(p^*-k)}.$$
In particular, choosing $\delta$ sufficiently small, we have $C_1>0$. Therefore, being $p<k<p^*$, there exist $\alpha, R>0$ with $R$ small enough,
 so that $J_A(u)\ge\alpha>0$
whenever $\|u\|=R$. Thus, condition $(ii)$ of Theorem \ref{mpthm} is satisfied.

Now, let $u\in X\setminus\{0\}$ then, using $(f_3)$ and $(K)$ we have
 
\begin{equation}\label{iii}
J_A(tu)\le \frac{t^p}{p}\|u\|^p -\frac{1}{p}\int_{\mathbb R^N}F(x,t^p|u|^p)dx
\le {t^p}\|u\|^p -Ct^q\int_{\mathbb R^N}|u|^q dx, 
\end{equation}
 
yielding immediately $J_A(tu)\to -\infty$ as $t\to \infty$, thanks to \eqref{green} and $q\in (p,p^*)$.
Thus, choosing $t_u>0$ such that $J_A(tu)<0$ for all $t\ge t_u$ and $\|t_u u\|>R$, then the proof of $(iii)$ is concluded.

Consider
$$\Gamma_u=\{\gamma\in C^0([0,1],X)\,/\, \gamma(0)=0\hbox{ and }\gamma(1)=t_u u\},$$
and
\begin{equation}\label{cu}
c_M=\inf_{\gamma\in\Gamma_u}\,\sup_{t\in [0,1]}J_A(\gamma(t)).
\end{equation}
Since the hypotheses of Theorem \ref{mpthm} are satisfied, there exists a $(PS)_{c_M}$ sequence.
\end{proof}

Now we claim that $J_A$ has the compactness property given by the validity of the $(PS)$ condition for levels $c$ below a suitable positive threshold.

\begin{lemma}\label{lem5}
Define 
\begin{equation}\label{csegnato}
c_P=\frac{S^{N/p}}{N\|K\|_\infty^{N/p^*}}.
\end{equation}
Then, the energy functional $J_A$ satisfies the $(PS)_c$ condition for every $c<c_P$.
\end{lemma}

\begin{proof}
Let $(u_n)_n$ be a $(PS)_c$ sequence in $X$, that is \eqref{PS_prop} holds. By Lemma \ref{lembound}, then $(u_n)_n$ is bounded in $X$. Thus, by \eqref{kato}, $(|u_n|)_n$ is bounded in $W^{1,p}_{V}(\mathbb R^N, \mathbb R)$, which is a reflexive Banach space, see \cite{AP}. 
By Banach-Alaoglu's Theorem,
 there exists $|u|\in W^{1,p}_{V}(\mathbb R^N, \mathbb R)$ 
 such that, up to subsequences, we get
\begin{enumerate}
\item[(I)] $|u_{n}|\rightharpoonup |u|$ in $W^{1,p}_{V}(\mathbb R^N, \mathbb R)$.
\item[(II)] Since $\nabla |u_{n}|\rightharpoonup \nabla |u|$ in $L^{p}(\mathbb{R}^N)$, the sequence of measures
 $(|\nabla |u_n||^p dx)_n$ is bounded and
$$|\nabla |u_{n}||^{p}dx\rightharpoonup \mu,$$
\item[(III)] analogously,
$$|u_n|^{p^*}dx\rightharpoonup \nu,$$
\end{enumerate}
where $\mu, \nu$ are bounded nonnegative measures on $\mathbb R^N$.
By Proposition \ref{lemben}, there exist at most countable set $J$, a family $(x_j)_{j\in J}$ of distinct points in
$\mathbb R^N$ and two families $(\nu_j)_{j\in J}, \,(\mu_j)_{j\in J}\in ]0,\infty[$ such that
\eqref{unpstarinf}, \eqref{dunpinf} hold, with $\nu_\infty$, $\mu_\infty$ defined in \eqref{munuinf}, satisfying
\begin{equation}\label{6.22}
S\nu_{j}^{p/p^{*}}\le\mu_{j}, \qquad S\nu_{\infty}^{p/p^*}\le \mu_\infty.
\end{equation}

Take a standard cut-off function $\psi\in C_{c}^{\infty}(\mathbb{R}^N)$, such that $0\le\psi\le1$ in $\mathbb{R}^N$,
$\psi=0$ for $|x|>1$, $\psi=1$ for $|x|\le 1/2$. For each index $j\in J$ and each $0<\varepsilon<1$, define
$$\psi_{\varepsilon}(x)=\psi\left(\frac{x-x_{j}}{\varepsilon}\right).$$

Since $J'_{A}(u_{n})\psi\to0$ being $(u_{n})_n$ a $(PS)_c$ sequence and 
choosing $\psi=\psi_\varepsilon u_n$, which is still bounded,  we have, as $n\to\infty$,
\begin{equation}\label{dis1}
\begin{aligned}
\Re \biggl(\int_{\mathbb R^N} &|\nabla_A u_n|^{p-2}\langle\nabla_A u_n,\nabla_A (u_n\psi_\varepsilon) \rangle_{\CC^N}  dx\biggr)+\int_{\mathbb R^N} V(x)|u_n|^{p} \psi_\varepsilon dx\\
&=\int_{\mathbb R^N}K(x)|u_n|^{p^*}\psi_\varepsilon dx+\int_{\mathbb R^N} f(x,|u_n|^p)|u_n|^p\psi_\varepsilon  dx+o(1).
\end{aligned}
\end{equation}
Now, using \eqref{for1} with $u=u_n$ and $\eta=\psi_\varepsilon$

$$\begin{aligned}
\Re &\biggl(\int_{\mathbb R^N} |\nabla_A u_n|^{p-2}\langle\nabla_A u_n,\nabla_A (u_n\psi_\varepsilon) \rangle_{\CC^N} dx\biggr)\\&=\Re \biggl(\int_{\mathbb R^N} |\nabla_A u_n|^{p-2}\nabla_A u_n\cdot (\overline{u_n}\nabla \psi_\varepsilon+\psi_\varepsilon\overline{\nabla_A u_n})dx\biggr)\\
&
=\int_{\mathbb R^N} |\nabla_A u_n|^p
\psi_\varepsilon dx+
\Re \biggl(\int_{\mathbb R^N}\overline{u_n}|\nabla_A u_n|^{p-2}\nabla_A u_n\cdot\nabla \psi_\varepsilon dx\biggr)\end{aligned}$$
 
so that, by \eqref{kato} and $(V)$, then \eqref{dis1} becomes 
\begin{equation}\label{nujbec}\begin{aligned}
&\int_{\mathbb R^N} |\nabla|u_n||^p
\psi_\varepsilon dx+
\Re \biggl(\int_{\mathbb R^N}\overline{u_n}|\nabla_A u_n|^{p-2}\nabla_A u_n\cdot\nabla \psi_\varepsilon dx\biggr) \\
&\le \int_{\mathbb R^N} |\nabla_A u_n|^p
\psi_\varepsilon dx+
\Re \biggl(\int_{\mathbb R^N}\overline{u_n}|\nabla_A u_n|^{p-2}\nabla_A u_n\cdot\nabla \psi_\varepsilon dx\biggr)+\int_{\mathbb R^N} V(x)|u_n|^{p} \psi_\varepsilon dx\\
&=\int_{\mathbb R^N}K(x)|u_n|^{p^*}\psi_\varepsilon dx+\int_{\mathbb R^N} f(x,|u_n|^p)|u_n|^p\psi_\varepsilon  dx+o(1).
\end{aligned}\end{equation}

Applying H\"older inequality, 
\begin{equation}\label{6.4}
\begin{aligned}
\biggl|\int_{\mathbb R^N}&\overline{u_n}|\nabla_A u_n|^{p-2}\nabla_A u_n\cdot\nabla \psi_\varepsilon dx\biggr|\\&\le \biggl(\int_{B_{\varepsilon}(x_j)\setminus B_{\varepsilon/2}(x_j)}|u_n|^p|\nabla \psi_\varepsilon|^pdx\biggr)^{1/p}\|\nabla_A u_n\|_p^{p-1} \\
&\le \| u_n\|^{p-1} \biggl(\int_{B_{\varepsilon}(x_j)}|u_n|^p|\nabla \psi_\varepsilon|^pdx\biggr)^{1/p}.
\end{aligned}
\end{equation}

Furthermore, by \eqref{compem} and (I), we have
$|u_n|\to |u|$ in
$L^{p}_\mathrm{loc}(\R^N,\R)$, yielding
$|u_n(x)|\to |u(x)|$ a.e.\ in $\omega=\overline{B}_{\varepsilon}(x_{j})$, up to subsequences,  and there exists $g\in L^p(\omega, \mathbb R)$
such that $|u_{n}(x)|\le g(x)$ a.e.\ in $\omega$.  Thus, $|u_n(x) \nabla\psi_{\varepsilon}(x)|\le C g(x)$
a.e.\ in $\omega$, as well as in $\mathbb R^N$, and in turn, Lebesgue dominated convergence theorem gives
\begin{equation}\label{conv_grad}
|u_n\nabla\psi_{\varepsilon}|\to|u\nabla\psi_{\varepsilon}| \ \textnormal{in} \ L^{p}(\mathbb{R}^N).
\end{equation}
Consequently, passing to the limit  for $n\to\infty$ in \eqref{6.4}, using \eqref{conv_grad}, the boundedness of $(|u_n|)_n\in W^{1,p}_{V}(\mathbb R^N, \mathbb R)$ by Lemma \ref{lembound} and \eqref{kato},
 H\"older's inequality with exponents
$N/(N-p)$ and $N/p$, we obtain
$$\begin{aligned}
&\limsup_{n\to\infty}\biggl|\int_{\mathbb R^N}\overline{u_n}|\nabla_A u_n|^{p-2}\nabla_A u_n\cdot\nabla \psi_\varepsilon dx\biggr|\\
&\le C \left(\int_{B_{\varepsilon}(x_{j})}|\nabla\psi_{\varepsilon}|^{N}dx\right)^{1/N}
\left(\int_{B_{\varepsilon}(x_{j})}|u|^{p^*}dx\right)^{1/p^*}\le C\left(\int_{B_{\varepsilon}(x_{j})}|u|^{p^{*}}dx
\right)^{1/p^{*}},
\end{aligned}$$
where in the last inequality we have used that $|\nabla \psi_\varepsilon|\le C \varepsilon^{-1}$ and $|B(x_i,2\varepsilon)| \le C'\varepsilon^N$.
In turn, letting $\varepsilon\to 0$ from $|u|\in L^{p^*}(\R^N,\R)$, we get
\begin{equation}\label{op0}
\lim_{\varepsilon\to 0}\limsup_{n\to\infty}\biggl|\int_{\mathbb R^N}\overline{u_n}|\nabla_A u_n|^{p-2}\nabla_A u_n\cdot\nabla \psi_\varepsilon dx\biggr|=0.
\end{equation}

Now, by $(f_1)$ we get $f(x,|u_n|^p)|u_n|^p\le h_1(x)|u_n|^p+h_2(x)|u_n|^k$ with $p<k<p^*$, 
in turn, having $\psi_\varepsilon$ compact support then \eqref{compem} is in force and  we can argue as above yielding the existence of $g_1^p\in L^{p^*/p}$  and $g_2^k\in L^{p^*/k}$ so that $f(x,|u_n|^p)|u_n|^p\le G\in L^1_{loc}(\mathbb R^N)$ and Lebesgue dominated convergence gives
\begin{equation}\label{pl}
\lim_{\varepsilon\to 0}\lim_{n\to\infty}\int_{\mathbb R^N}f(x,|u_n|^p)|u_n|^p\psi_\varepsilon  dx=\lim_{\varepsilon\to 0}\int_{\mathbb R^N}f(x,|u|^p)|u|^p\psi_\varepsilon  dx=0.
\end{equation}
Then, by \eqref{pl} and \eqref{op0}, from \eqref{nujbec} we can conclude for $n$ large
\begin{equation}\label{nabla_le_K_ustar}\int_{\mathbb R^N} |\nabla|u_n||^p
\psi_\varepsilon dx\le\int_{\mathbb R^N}K(x)|u_n|^{p^*}\psi_\varepsilon dx+o(1).\end{equation}
Hence, since $K\in C(\mathbb R^N)$ and $\mbox{supp} \psi_\varepsilon=B_\varepsilon(x_j)$, if $\varepsilon\to0$ we deduce
\begin{equation}\label{6.23}
\mu_{j}\le K(x_{j})\nu_{j}.
\end{equation}
Consequently, either  $\nu_j=0$ and then also $\mu_j=0$,  or $\nu_j>0$. We will show that  
the latter case cannot occur for each $j\in J$, with $J$ given in Lemma \ref{lem4.1}. First observe that  combining \eqref{6.23} and \eqref{6.22}, we have
\begin{equation}\label{sxj}
S\leq K(x_j)\nu_j^{p/N}.
\end{equation}
The inequality above establishes that the concentration of the measure $\nu$ can occur only at points $x_j$ where $K(x_j)>0$.
Consequently, from \eqref{6.22} and \eqref{6.23} the measure $\mu$ can concentrate at points
in which the measure $\nu$ can.
Hence, the set $X_J=\{x_j: j\in J\}$  does not contain those points $x_{j}$ which are zeros for $K$.

Let $J_1=\{j\in J: K(x_j)>0\}$, we claim that $J_1=\emptyset$. To prove the claim, we proceed by contradiction, noting that  for any $j\in J_1$, then \eqref{sxj} implies
\begin{equation}\label{J2}
 \nu_{j}\ge \left( \frac{S}{K(x_{j})}\right)^{N/p}\ge \left( \frac{S}{\|K\|_\infty}\right)^{N/p},
\end{equation}
 which in particular gives that
 $|J_1|<\infty$ being $\nu$ a bounded measure; indeed, from \eqref{unpstarinf} and \eqref{J2},
 we get 
$$\infty>\int_{\mathbb{R}^N}d\nu=\|u\|_{p^*}^{p^*}+\int_{\mathbb R^N}\sum_{j\in J_1}\nu_j\delta_{x_j}dx+\nu_\infty
\ge \|u\|_{p^*}^{p^*}+\left(\frac{S}{\|K\|_\infty}\right)^{N/p}|J_1|+\nu_\infty.$$
 Now, we show that \eqref{J2} cannot occur. 
Using $\psi_\varepsilon\le 1$ and $(f_2)$, being $\theta>p$, we get
\begin{equation}\label{epn}\begin{aligned}
c+o(1)&= J_A(u_n)-\frac{1}{p}\langle J_A'(u_n),u_n\rangle= \frac1N\int_{\mathbb R^N}K(x)|u_n|^{p^*}dx\\
&\qquad+\frac1p \int_{\mathbb R^N}\left[ f(x,|u_n|^p)|u_n|^{p}-F(x,|u_n|^p)\right] dx
\\&\ge \frac1N\int_{\mathbb R^N}K(x)|u_n|^{p^*}\psi_\varepsilon dx.
\end{aligned}\end{equation}

 Using \eqref{J2} and letting $n\to\infty$ and $\varepsilon\to0$, we obtain
$$c\ge \frac 1N \nu_j K(x_j)\ge \frac{S^{N/p}}{N\|K\|_\infty^{N/p^*}}(=c_P)$$
which contradicts the assumption $c<c^*$, thus $J_1 = \emptyset$. It remains to show that the concentration
of $\nu$ cannot occur at infinity, namely $\nu_\infty=0$. 

Following the same idea used to prove \eqref{6.23} but with the following cutoff function   $\psi_{R}\in C^{\infty}(\mathbb{R}^N)$
such that $0\le\psi_{R}\le1$ in $\mathbb{R}^{N}$, $\psi_R(x)=0$ for $|x|<R$ and $\psi_{R}(x)=1$ for $|x|>2R$, so that \eqref{nabla_le_K_ustar} holds with $\psi_\varepsilon$ replaced by $\psi_R$.
Noting that
$$\lim_{R\to\infty}\limsup_{n\to\infty}\left\{\int_{\mathbb{R}^N}K|u_{n}|^{p^{*}}\psi_{R}dx\right\}
\le\|K\|_{\infty}\nu_{\infty},$$
we finally get
$\|K\|_\infty\nu_\infty\ge\mu_\infty$,
which, together with \eqref{6.22} gives $\nu_{\infty}\ge  S^{N/p}\|K\|_\infty^{-N/p}$.
As for \eqref{epn}, we get again a contradiction. 
Consequently,
$$\lim_{n\to\infty}\int_{\mathbb{R}^N}|u_{n}|^{p^{*}}dx=\int_{\mathbb{R}^N}|u|^{p^{*}}dx,$$
that is $\|u_n\|_{p^*}\to\|u\|_{p^*}$ as $n\to\infty$, which combined with $|u_n(x)|\to |u(x)|$ a.e. in $\R^N$, the latter obtained by an exhaustion process applied to a.e. convergence on compact sets in $\mathbb R^N$ by \eqref{compem},  implies $\|u_n-u\|_{p^*}\to 0$ by Brezis Lieb Lemma in \cite{bl}. It remains to prove 
\begin{equation}\label{claimfin}
\|u_n-u\|\to0, \quad \text{as}\,\,\,n\to\infty,
\end{equation}
 which is equivalent to prove
\begin{equation}\label{claimfin_equiv}
\int_{\mathbb R^N}|\nabla_A (u_n-u)|^pdx, \,\, \int_{\mathbb R^N}V(x)|u_n-u|^p\to 0 
\quad\mbox{as}\quad n\to\infty.\end{equation}
To this aim, since $(u_n)_n$ is a $(PS)_c$ sequence, we have
\begin{equation}\label{g-11}\begin{aligned}
o(1)&=\langle J'_A(u_n)-J'_A(u),u_n-u\rangle\\
&=\Re \!\int_{\mathbb R^N} \bigl\langle|\nabla_A u_n|^{p-2}\nabla_A u_n-|\nabla_A u|^{p-2}\nabla_A u, \nabla_A (u_n-u)\bigr\rangle_{\mathbb C^N}dx
\\ &\quad+\Re\!\int_{\mathbb R^N} V(x)\langle|u_n|^{p-2}u_n-|u|^{p-2}u,u_n-u\rangle_{\mathbb C^N}dx\\&\quad-\Re\int_{\mathbb{R}^N}K(x)\langle |u_n|^{p^*-2}u_n-|u|^{p^*-2}u, u_n-u\rangle_{\mathbb C^N}dx\\ &\quad
-\Re\!\int_{\mathbb{R}^N}\langle f(x, |u_n|^p)|u_n|^{p-2}u_n-f(x, |u|^p)|u|^{p-2}u, u_n-u\rangle_{\mathbb C^N}dx.
\end{aligned}\end{equation}
By Lemma \ref{lembound}, using H\"older's and Schwarz's inequality, the function $g_1$ above \eqref{pl} and the convergence of $(u_n)_n$ in $L^{p^*}(\R^N)$ we get
 
$$\begin{aligned}
\lefteqn{\left|\int_{\mathbb{R}^N}K(x)\langle |u_n|^{p^*-2}u_n-|u|^{p^*-2}u, u_n-u\rangle_{\mathbb C^N}dx\right|}
\\
&\le \|K\|_\infty \int_{\mathbb{R}^N} \left(|u_n|^{p^*-1}+|u|^{p^*-1}\right)|u_n-u|dx\\
& \le \|K\|_\infty \left(\|u_n\|_{p^*}^{p^*-1}+\|u\|_{p^*}^{p^*-1}\right)\|u_n-u\|_{p^*}=o(1).
\end{aligned}$$
 
Similarly, by using $(f_1)$, we have
$$\begin{aligned}
&\left|\int_{\mathbb{R}^N}\langle f(x, |u_n|^p)|u_n|^{p-2}u_n-f(x, |u|^p)|u|^{p-2}u, u_n-u\rangle_{\mathbb C^N}\right|\\
&\quad \le\int_{\mathbb{R}^N} \left[h_1(x)|u_n|^{p-1}+h_2(x)|u_n|^{k-1}\right]|u_n-u|dx\\
& \hspace{2cm}+
 \int_{\mathbb{R}^N} \left[h_1(x)|u|^{p-1}+h_2(x)|u|^{k-1}\right]|u_n-u|dx\\
 & \quad \le\left(\|h_1\|_{N/p}\left[\|u_n\|_{p^*}^{p-1}+\|u\|_{p^*}^{p-1}\right ]+\|h_2\|_{p^*/(p^*-k)}\left[\|u_n\|_{p^*}^{k-1}+\|u\|_{p^*}^{k-1}\right ]\right)\|u_n-u\|_{p^*}=o(1).
\end{aligned}$$
Thus, \eqref{g-11} reduces to 
\begin{equation}\label{crucial}\begin{aligned}
o(1)&=\langle J'_A(u_n)-J'_A(u),u_n-u\rangle\\
&=\Re \!\int_{\mathbb R^N} \bigl\langle|\nabla_A u_n|^{p-2}\nabla_A u_n-|\nabla_A u|^{p-2}\nabla_A u, \nabla_A (u_n-u)\bigr\rangle_{\mathbb C^N}dx
\\ &\quad+\Re\!\int_{\mathbb R^N} V(x)\langle|u_n|^{p-2}u_n-|u|^{p-2}u,u_n-u\rangle_{\mathbb C^N}dx,
\end{aligned}\end{equation}
 where we recall that $\Re\langle|b|^{p-2}b-|a|^{p-2}a,
b-a\rangle_{\CC^N}\ge0$ for all $a,b\in\CC^N$ and $p>1$ (cfr. Appendix A), so that
\eqref{crucial} gives
\begin{equation}\label{final_o(1)}
    \begin{aligned}
&\Re \!\int_{\mathbb R^N} \bigl\langle|\nabla_A u_n|^{p-2}\nabla_A u_n-|\nabla_A u|^{p-2}\nabla_A u, \nabla_A (u_n-u)\bigr\rangle_{\mathbb C^N}dx=o(1)\quad\mbox{as } n\to\infty,
\\ &\Re\!\int_{\mathbb R^N} V(x)\langle|u_n|^{p-2}u_n-|u|^{p-2}u,u_n-u\rangle_{\mathbb C^N}dx=o(1) \quad\mbox{as } n\to\infty.
\end{aligned}
\end{equation}
 In what follows, we make use of the following complex version of Simon's inequality, see \cite{simon}, valid for all
$a,b\in\mathbb C^N$
\begin{equation}\label{diaz_complex}
|a-b|^{p}\leq c\begin{cases} \, \Re\langle |a|^{p-2}a- |b|^{p-2}b
,a-b\rangle_{\mathbb C^N}\quad&\mbox{for }\phantom{1<\,}p\geq 2;\\
 \left(\Re\langle |a|^{p-2}a- |b|^{p-2}b
,a-b\rangle_{\mathbb C^N}\right)^{p/2}\left(|a|^p+|b|^p\right)^{(2-p)/2} 
\quad&\mbox{for }1<p<2,\end{cases}\end{equation}
 for details, we refer to Appendix \ref{ineq}.

In particular, if $1<p<2$, then, by \eqref{diaz_complex} applied with $a=\nabla_A u_n$ and $b=\nabla_A u$ and by Holder's inequality with exponents $2/p$
and $2/(2-p)$,  we arrive to 
$$\begin{aligned}\int_{\mathbb R^N}&|\nabla_A(u_n-u)|^pdx
\\&\quad\le c\, \!\left(\int_{\mathbb R^N} \! \Re\bigl\langle|\nabla_A u_n|^{p-2}\nabla_A u_n-|\nabla_A u|^{p-2}\nabla_A u, \nabla_A (u_n-u)\bigr\rangle_{\mathbb C^N}dx\right)^{p/2}\\&\hspace{2cm}\cdot
\bigl(\|u\|^p+\|u_n\|^p\bigr)^{\frac{2-p}2}dx
\\&\quad  \le c\, \biggl(\!\int_{\mathbb R^N} \!  \Re \bigl\langle|\nabla_A u_n|^{p-2}\nabla_A u_n-|\nabla_A u|^{p-2}\nabla_A u, \nabla_A (u_n-u)\bigr\rangle_{\mathbb C^N}
dx\biggr)^{p/2}\\&\hspace{2cm}\cdot \left(\|u\|^{p(2-p)/2}+\|u_n\|^{p(2-p)/2}\right),
\end{aligned}$$
 where we have used the standard inequality with $\alpha,\beta>0$
\begin{equation}\label{standard}
(\alpha+\beta)^\gamma\le \overline C(\alpha^\gamma+\beta^\gamma), \qquad \text{with}\quad \overline C=\max\{2^{\gamma-1},1\}.
\end{equation}
By the boundedness of $(u_n)_n$ as $(PS)$ sequence  by Lemma \ref{lembound}, in turn 
$$ \int_{\mathbb R^N}|\nabla_A(u_n-u)|^pdx \le C \biggl(\!\int_{\mathbb R^N} \Re\bigl\langle|\nabla_A u_n|^{p-2}\nabla_A u_n-|\nabla_A u|^{p-2}\nabla_A u, \nabla_A (u_n-u)\bigr\rangle^{p}_{\mathbb C^N}
dx\biggr)^{p/2},
$$ where $C$ does not depend on $n$.
On the other hand, in the same way as above, applying \eqref{diaz_complex}  with $a=u_n$ and $b=u$, for $1<p<2$ it holds
$$\int_{\mathbb R^N} V(x)|u_n-u|^pdx\le  c \biggl(\Re\!\int_{\mathbb R^N} V(x)\langle|u_n|^{p-2}u_n-|u|^{p-2}u,u_n-u\rangle_{\mathbb C^N}dx\biggr)^{p/2}.$$
In turn, \eqref{final_o(1)} immediately gives \eqref{claimfin_equiv} when $1<p<2$.
While, if $p\ge 2$, by \eqref{diaz_complex} we immediately have
$$|\nabla_A(u_n-u)|^pdx\le c \!\int_{\mathbb R^N}   \Re\bigl\langle|\nabla_A u_n|^{p-2}\nabla_A u_n-|\nabla_A u|^{p-2}\nabla_A u, \nabla_A (u_n-u)\bigr\rangle^{p}_{\mathbb C^N} dx$$
and 
$$\int_{\mathbb R^N} V(x)|u_n-u|^pdx\le c 
\Re\!\int_{\mathbb R^N} V(x)\langle|u_n|^{p-2}u_n-|u|^{p-2}u,u_n-u\rangle_{\mathbb C^N}dx,$$
so that \eqref{claimfin_equiv} with $p\ge2$ follows by \eqref{final_o(1)}. In conclusion,  \eqref{claimfin} is in force concluding the proof of the lemma.
\end{proof}

From now on we denote,
\begin{equation}\label{clambda}
c_A = \inf_{u\in X\setminus\{0\}} \max_{t\ge 0} J_A(tu).
\end{equation}
\begin{remark}\label{culambda}
Obviously, $c_A\geq c_M$, where $c_M$ is defined in \eqref{cu},
since $J_A(tu)<0$ for $u\in X\setminus\{0\}$ and $t$ large by the structure of $J_A$.
\end{remark}
\begin{lemma}\label{c<csegnato}
Let $c_P$ and $c_A$ be defined as in \eqref{csegnato} and \eqref{clambda}, respectively. Then, 
$$0< c_A<c_P.$$
\end{lemma}

\begin{proof}
It is enough to prove that there exists $v\in X\setminus\{0\}$ 
such that ${\max_{t\ge 0} J_A(tv)<c_P}$. In turn, passing to the infimum we get the assertion of the lemma.

Let $\theta$ be a linear real function  defined in $\mathbb R^N$ by $\theta(x)=- A(0) \cdot x=-\sum_{k=1}^NA_k(0)x_k$. Note that there exists $\delta_A>0$ sufficiently small such that 
\begin{equation}\label{def_small_c}|(A +\nabla\theta)(x)|^2<c \quad\text{for all }|x| <\delta_A\end{equation}
being $(A +\nabla\theta)(0) =0$ and by the continuity of $A$, thanks to  $(A)$. 

Now consider $\psi\in C^1$ defined in \eqref{defcut}, take a smaller $\delta_\psi$ if needed such that $\delta_\psi<\min\{\delta_A, \delta_K\}$, where $\delta_K$ is defined in $(K)$, so that $\delta=\delta_\psi\in(0,\frac{\delta_A}{\varepsilon})$ 
and let $u_\varepsilon(x)=e^{i\theta(x)} w_\varepsilon(x)$, where $w_\varepsilon$ is defined in \eqref{defwtil} with $0<\varepsilon<1$. In particular $|u_\varepsilon|=|w_\varepsilon|$. 
By using \eqref{standard},
we compute
\begin{equation}\label{normapeps}\begin{aligned}
\| u_\varepsilon\|^p&=\int_{\mathbb R^N} [|\nabla_A u_\varepsilon|^p+V(x)|u_\varepsilon|^p] dx\\
&= \int_{\mathbb R^N}[|\nabla w_\varepsilon+i\nabla\theta(x) w_\varepsilon+iA(x)w_\varepsilon|^{p}+V(x)|w_\varepsilon|^p] dx\\
&=\int_{\mathbb R^N}[\left(|\nabla w_\varepsilon|^2+|\nabla\theta(x)+A(x)|^2|w_\varepsilon|^2\right)^{p/2}+V(x)|w_\varepsilon|^p] dx\\
&\le \overline C\left[\|\nabla w_\varepsilon\|_p^p+(c+\|V\|_{L^\infty(B_\delta)})\|w_\varepsilon\|_p^p\right].
\end{aligned}\end{equation}
Thus \eqref{normapeps} and Remark \ref{weW1p} imply $u_\varepsilon\in X$.

Arguing as in \eqref{iii}, by $(K)$, $(f_3)$, \eqref{normapeps}, we have
$$\begin{aligned}
J_A(tu_\varepsilon)
\le \frac{t^p}{p}\|u_\varepsilon\|^p-\frac{\lambda }{p}t^q\|u_\varepsilon\|_q^q\to-\infty,
\end{aligned}$$
as $t\to\infty$, being $q>p$.

So that by Lemma \ref{mpt_geometry}, 
there exists $t_\varepsilon>0$ such that $J_A(t_\varepsilon u_\varepsilon)={\max_{t\ge 0} J_A(tu_\varepsilon)}$. Thus, by the regularity of $J_A$, we have $\langle J'_A(t_\varepsilon u_\varepsilon), u_\varepsilon\rangle=0$, or equivalently
\begin{equation}\label{E'c}
\| u_\varepsilon\|^p
=t_\varepsilon^{p^*-p} \int_{\mathbb R^N}K(x)| u_\varepsilon|^{p^*} dx+\int_{\mathbb R^N} f(x,|t_\varepsilon u_\varepsilon|^p)| u_\varepsilon|^p dx.
\end{equation}
Let us show that there exist $B_1, B_2>0$ such that
\begin{equation}\label{3.6}
B_1 \le t_\varepsilon \le B_2 \quad \text{for}\quad \varepsilon > 0 \quad \text{sufficiently small.}
\end{equation}
From \eqref{E'c}, we have  by $(f_2)$
\begin{equation}\label{tinf}
\| u_\varepsilon\|^p \ge t_\varepsilon^{p^*-p} \int_{\mathbb R^N}K(x)| u_\varepsilon|^{p^*} dx.
\end{equation}
If $t_\varepsilon\to \infty$ as $\varepsilon\to 0$, we immediately reach a contradiction since the right-hand side of \eqref{tinf} is unbounded being, by $(K)$ 
$$0<\!\int_{\mathbb R^N}K(x)|u_\varepsilon|^{p^*} dx\le \|K\|_\infty \|w_\varepsilon\|_{p^*}^{p^*}=\|K\|_\infty$$
by the definition of $w_\varepsilon$,
while the left-hand side is bounded by \eqref{normapeps} and Lemma \ref{tallemma}. In turn, $t_\varepsilon$ has to be bounded from above.

Now, we assume $t_\varepsilon\to 0$ for $\varepsilon\to 0$, up to subsequence. 
By \eqref{f0f1}, for any $\xi>0$ there exists $C_\xi>0$ such that
 $$\begin{aligned}
\int_{\mathbb R^N}& f(x,|t_\varepsilon u_\varepsilon|^p)| u_\varepsilon|^p dx\le \int_{\mathbb R^N} \biggl(\xi | u_\varepsilon|^p +C_\xi\left[h_3(x)|t_\varepsilon u_\varepsilon|^{k-p}\right]| u_\varepsilon|^p \biggr)dx
\\&\le \xi \|u_\varepsilon\|_p^p+ C_\xi t_\varepsilon ^{k-p}\|h_3\|_{p^*/(p^*-k)}\|u_\varepsilon\|_{p^*}^k 
\le \frac{\xi}{V_0} \|u_\varepsilon\|^p+ C_\xi t_\varepsilon ^{k-p}\|h_3\|_{p^*/(p^*-k)}\|u_\varepsilon\|_{p^*}^k, 
 \end{aligned}$$
 where in the last inequality we have used $(V)$. Now choosing $\xi=V_0/2$, thanks to
\eqref{E'c}, we have
$$\begin{aligned}
\frac 12\| u_\varepsilon\|^p
&\le t_\varepsilon^{p^*-p}\|K\|_\infty  \|u_\varepsilon\|_{p^*}^{p^*} + C_\xi t_\varepsilon ^{k-p}\|h_3\|_{p^*/(p^*-k)}\|u_\varepsilon\|_{p^*}^k
\end{aligned}$$
which leads to an absurd, taking $t_\varepsilon\to 0$ for $\varepsilon\to 0$, since $p<k$ and the left-hand side is far from $0$. Therefore, \eqref{3.6} holds true.
In order to reach our claim, we need to prove that
\begin{equation}\label{claimfinale}
J_A(t_\varepsilon u_\varepsilon)\le \frac{S^{N/p}}{N\|K\|_\infty^{N/p^*}}(=c_P),
\end{equation}
to this aim we have to divide the proof into two cases according to the range of $p$.

If $1<p<2$, from \eqref{normapeps}, with $\overline C=1$, and Lemma 1 we get 
$$\begin{aligned}
\|u_\varepsilon\|^p 
\le \|\nabla w_\varepsilon\|_p^p&+(c+\|V\|_{L^\infty(B_\delta)})\|w_\varepsilon\|_p^p\\
&=S+O(\varepsilon^{\frac{N-p}{p-1}})+\begin{cases}O(\varepsilon^p) &\text{if}\,\,N>p^2,\\
O(\varepsilon^{\frac{N-p}{p-1}}|\log\varepsilon|)&\text{if}\,\,N=p^2,\\
O(\varepsilon^{\frac{N-p}{p-1}}) &\text{if}\,\,N<p^2.
\end{cases}
\end{aligned}$$
Consequently, for $1<p<2$, it holds 
\begin{equation}\label{normapeps2}
\|u_\varepsilon\|^p=S+\mathcal O_{\varepsilon, 1}, \qquad \text{where}\quad \mathcal O_{\varepsilon, 1}=\begin{cases}O(\varepsilon^p) &\text{if}\,\,N>p^2,\\
O(\varepsilon^{p}|\log\varepsilon|)&\text{if}\,\,N=p^2,\\
O(\varepsilon^{\frac{N-p}{p-1}}) &\text{if}\,\,N<p^2.
\end{cases}
\end{equation}
Taking inspiration from \cite{dh}, by $(f_3)$, we have $$\begin{aligned}
J_A(t_\varepsilon u_\varepsilon)&\le\frac{t_\varepsilon^p}{p}\|u_\varepsilon\|^p-\frac{\lambda}{p}t_\varepsilon^q\|u_\varepsilon\|_q^q-\frac{t_\varepsilon^{p^*}}{p^*}\int_{\mathbb R^N}K(x)|u_\varepsilon|^{p^*}dx=E(\varepsilon)+F(\varepsilon)
\end{aligned}$$
where
\begin{equation}\label{defE}
E(\varepsilon)=\frac{t_\varepsilon^p}{p}\|u_\varepsilon\|^p-\frac{K(0)t_\varepsilon^{p^*}}{p^*},
\end{equation}
\begin{equation}\label{defF}
F(\varepsilon)=-\frac{\lambda}{p}t_\varepsilon^q\|u_\varepsilon\|_q^q+\frac{t_\varepsilon^{p^*}}{p^*}\int_{\mathbb R^N}(K(0)-K(x))|u_\varepsilon|^{p^*}dx.
\end{equation}
Before going ahead note that, for any $D_1, D_2>0$, one has
\begin{equation}\label{d1d2}
\frac{t^p}{p}D_1-\frac{t^{p^*}}{p^*}D_2\le \frac{1}{N} \left(\frac{D_1}{D_2^{(N-p)/N}}\right)^{N/p} \quad \text{for all} \,\,t\ge 0.
\end{equation}
Then \eqref{d1d2}, \eqref{normapeps2} and $K(0)=\|K\|_\infty$ by $(K)$, all applied to \eqref{defE}, give 
\begin{equation}\label{EE}
E(\varepsilon)\le\frac{1}{N}\left(\frac{\|u_\varepsilon\|^p}{K(0)^{(N-p)/N}}\right)^{N/p}\le c_P+\mathcal O_{\varepsilon, 1},
\end{equation}
where we have used
\begin{equation}\label{Oe1}\begin{aligned}
(S+\mathcal O_{\varepsilon, 1})^{N/p}&\le S^{N/p}+\frac Np(S+\mathcal O_{\varepsilon, 1})^{(N-p)/p}\mathcal O_{\varepsilon, 1}
=S^{N/p}+ \mathcal O_{\varepsilon, 1},
\end{aligned}\end{equation}
 since $N>p$, in light of the following inequality
\begin{equation}\label{inr}
(a+b)^r\le a^r+r(a+b)^{r-1}b,\qquad \text{for all}\quad a,b>0,\,\, r\ge 1.
\end{equation}
On the other hand, \eqref{3.6} together with the definition of $u_\varepsilon$, turns \eqref{defF} into
$$\begin{aligned}
F(\varepsilon)&\le -\frac{\lambda }{p}B_1^q\|w_\varepsilon\|_q^q
+\frac{B_2^{p^*}c_{p,N}^{p^*}}{p^*\|\psi U_\varepsilon\|_{p^*}^{p^*}}\int_{\R^N}(K(0)-K(x))\psi(x)^{p^*}\left[\frac{\varepsilon^{\frac{1}{p-1}}}{\varepsilon^{\frac{p}{p-1}}+|x|^{\frac{p}{p-1}}}\right]^N.
\end{aligned}$$
By using (K) since $\delta<\delta_K$, the change of variables $x=\varepsilon y$, and passing to radial coordinates, we get
$$\begin{aligned}
\int_{\R^N}&(K(0)-K(x))\psi(x)^{p^*}\left[\frac{\varepsilon^{\frac{1}{p-1}}}{\varepsilon^{\frac{p}{p-1}}+|x|^{\frac{p}{p-1}}}\right]^N dx\sim \int_{B_\delta}|x|^\tau\psi(x)^{p^*}\frac{\varepsilon^{\frac{N}{p-1}}}{\bigl(\varepsilon^{\frac{p}{p-1}}+|x|^{\frac{p}{p-1}}\bigr)^N} dx\\
&\le 
\int_{\R^N}|\varepsilon y|^\tau \frac{\varepsilon^{\frac{N}{p-1}+N}}{\bigl(\varepsilon^{\frac{p}{p-1}}+|\varepsilon y|^{\frac{p}{p-1}}\bigr)^N}
dy\\
&=\varepsilon^\tau\int_{\R^N}\frac{|y|^\tau}{(1+|y|^{\frac{p}{p-1}})^N} dy=C\varepsilon^\tau\int_0^\infty\frac{r^{\tau+N-1}}{(1+r^{\frac{p}{p-1}})^N} dr=O(\varepsilon^\tau),
\end{aligned}$$
since $0<\tau<N/(p-1)$.
Moreover, since $U_\varepsilon\in L^{p^*}(\mathbb R^N)$ and  $\|\psi U_\varepsilon\|_{p^*}^{p^*}$
does not depend on $\varepsilon$ by \eqref{Uepsx0}, then it is bounded  away from $0$ yielding
\begin{equation}\label{FF}
F(\varepsilon)\le -\frac{\lambda }{p}B_1^q\|w_\varepsilon\|_q^q+O(\varepsilon^{\tau}).
\end{equation}
So that, by the choice of $\tau$ in (K), from \eqref{EE} and \eqref{FF}, we get for $1<p<2$
\begin{equation}\label{OOe1} 
J_A(t_\varepsilon u_\varepsilon)\le c_P+\mathcal O_{\varepsilon, 1}-\frac{\lambda }{p}B_1^q\|w_\varepsilon\|_q^q.
\end{equation}

If $p\ge2$ then \eqref{normapeps} holds $\overline C=2^{p/2-1}$ but to proceed we need no coefficient (different from $1$)  in the term $\|\nabla w_\varepsilon\|_p^p$. For this reason, we proceed by using \eqref{normapeps}  in \eqref{inr} with $r=p/2  (\ge1)$, we get
$$\begin{aligned}
\|u_\varepsilon\|^p&=\int_{\mathbb R^N}\left(|\nabla w_\varepsilon|^2+|\nabla\theta(x)+A(x)|^2|w_\varepsilon|^2\right)^{p/2}\\&\le\int_{\mathbb R^N}\left[|\nabla w_\varepsilon|^p+c\frac p2\left( |\nabla w_\varepsilon|^2+c|w_\varepsilon|^2\right)^{p/2-1}|w_\varepsilon|^2\right] dx\\&
\le\|\nabla w_\varepsilon\|_p^p
+c\frac p2 \overline{C}\int_{\mathbb R^N}[  |\nabla w_\varepsilon|^{p-2}|w_\varepsilon|^2+c|w_\varepsilon|^p] dx,
\end{aligned}$$
where $c$ is given in \eqref{def_small_c} and  in the last inequality we have used \eqref{standard} with $\gamma=p/2-1$. In turn,  by $(V)$ and Holder's inequality, we have 
$$\begin{aligned}
\|u_\varepsilon\|^p&\le 
\|\nabla w_\varepsilon\|_p^p
+c\frac p2 \overline{C}(\|\nabla w_\varepsilon\|_p^{p-2}\cdot\|w_\varepsilon\|_p^2+c\|w_\varepsilon\|_p^p)+\|V\|_{L^\infty(B_\delta)}\|w_\varepsilon\|_p^p\\
&\le \|\nabla w_\varepsilon\|_p^p+C(\|\nabla w_\varepsilon\|_p^{p-2}\cdot\|w_\varepsilon\|_p^2+\|w_\varepsilon\|_p^p)\\
&=\|\nabla w_\varepsilon\|_p^p+C\|w_\varepsilon\|_p^p+(S+O(\varepsilon^{\frac{N-p}{p-1}}))^{(p-2)/p}\cdot\begin{cases}O(\varepsilon^2) &\text{if}\,\,N>p^2,\\
O(\varepsilon^{\frac{2(N-p)}{p(p-1)}}|\log\varepsilon|^{2/p})&\text{if}\,\,N=p^2,\\
O(\varepsilon^{\frac{2(N-p)}{p(p-1)}}) &\text{if}\,\,N<p^2,
\end{cases}\\
&=\|\nabla w_\varepsilon\|_p^p+\begin{cases}O(\varepsilon^p)+O(\varepsilon^2)
&\text{if}\,\,N>p^2,\\
O(\varepsilon^{\frac{N-p}{p-1}}|\log\varepsilon|)+O(\varepsilon^{\frac{2(N-p)}{p(p-1)}}|\log\varepsilon|^{2/p})
&\text{if}\,\,N=p^2,\\
O(\varepsilon^{\frac{N-p}{p-1}}) +O(\varepsilon^{\frac{2(N-p)}{p(p-1)}})
&\text{if}\,\,N<p^2,
\end{cases}\\
&=S+O(\varepsilon^{\frac{N-p}{p-1}})+\begin{cases}O(\varepsilon^2) &\text{if}\,\,N>p^2,\\
O(\varepsilon^{\frac{2(N-p)}{p(p-1)}}|\log\varepsilon|^{2/p})&\text{if}\,\,N=p^2,\\
O(\varepsilon^{\frac{2(N-p)}{p(p-1)}}) &\text{if}\,\,N<p^2,
\end{cases}\end{aligned}$$
where the estimates above follows from Lemma \ref{tallemma}. In turn, for $p\ge2$, we arrive to
\begin{equation}\label{est_norm_pge2}\|u_\varepsilon\|^p=S+\mathcal O_{\varepsilon, 2},\qquad\quad \mathcal O_{\varepsilon, 2}=\begin{cases}O(\varepsilon^2) &\text{if}\,\,N>p^2,\\
O(\varepsilon^2|\log\varepsilon|^{2/p})&\text{if}\,\,N=p^2,\\
O(\varepsilon^{\frac{2(N-p)}{p(p-1)}}) &\text{if}\,\,N<p^2.
\end{cases}
\end{equation}
Similarly to \eqref{Oe1}, by using \eqref{inr}, we obtain
$(S+\mathcal O_{\varepsilon, 2})^{N/p}\le S^{N/p}+C \mathcal O_{\varepsilon, 2}$. 
Hence, analogously to the case $1<p<2$, by \eqref{d1d2}, \eqref{est_norm_pge2}, $(f_3)$ and \eqref{3.6} we get for $p\ge2$
\begin{equation}\label{OOe2} 
\begin{aligned}
J_A(t_\varepsilon u_\varepsilon)&\le c_P+\mathcal O_{\varepsilon, 2}-\frac{\lambda }{p}B_1^q\|w_\varepsilon\|_q^q. \\
\end{aligned}\end{equation}

In order to understand the behaviour of the last two terms in \eqref{OOe1} and \eqref{OOe2}, following the standard technique in \cite{ajmaa, dh, Peralcourse, SY}, we have to divide our analysis into different cases. 
 
{\bf Case $N>p^2$}. In this situation, then $q>p>N(p-1)/(N-p)$. Thus,  combining either \eqref{OOe1} or 
\eqref{OOe2} with \eqref{est_norm_q} in Lemma \ref{tallemma}, we infer 
$$J_A(t_\varepsilon u_\varepsilon)\le c_P+ O(\varepsilon^{\nu})-\lambda O(\varepsilon^{N-\frac{N-p}{p}q}),\qquad \nu=\min\{2,p\}.$$
If $1<p\le 2$,  since $N-\frac{N-p}{p}q<p$ by $q>p$, we have by \eqref{OOe1} and \eqref{est_norm_q}
$$J_A(t_\varepsilon u_\varepsilon)\le c_P-\lambda O(\varepsilon^{N-\frac{N-p}{p}q}).$$
While, if $p>2$, by \eqref{OOe2} and \eqref{est_norm_q}, we can choose $\lambda=\varepsilon^{-\sigma}$, $\sigma>0$, so that 
$$J_A(t_\varepsilon u_\varepsilon)\le c_P+O(\varepsilon^2)-O(\varepsilon^{N-\frac{N-p}{p}q-\sigma})=c_P-O(\varepsilon^{N-\frac{N-p}{p}q-\sigma})$$
by choosing $\sigma>0$ such that
$$N-2-\frac{N-p}{p}q<\sigma<N-\frac{N-p}{p}q$$
that is
$$0<N-\frac{N-p}{p}q-\sigma<2.$$
Thus, for any $1<p<N$, we have obtained \eqref{claimfinale}
for $\varepsilon>0$ sufficiently small and $\lambda$ large only for $p>2$.

\noindent
{\bf Case $N=p^2$}. By using \eqref{est_norm_q}, \eqref{OOe1} and
\eqref{OOe2}, since $N(p-1)/(N-p)=p<q$ we have, 
$$\begin{aligned}J_A(t_\varepsilon u_\varepsilon)&\le 
c_P+ O(\varepsilon^{\nu}|\log(\varepsilon)|^{\frac{\nu}{p}})-\lambda O(\varepsilon^{p^2-(p-1)q}).
\end{aligned}$$
If $1<p<2$, then $\nu=p$ and we have
$$\lim_{\varepsilon\to 0^+}\varepsilon^{(q-p)(p-1)}|\log(\varepsilon)|=0.$$
While, if $p>2$, then $\nu=2$ and we can choose $\lambda=\varepsilon^{-\sigma}$, $\sigma>0$,  such that
$$-2+p^2-(p-1)q<\sigma<p^2-(p-1)q,$$
thus  
$$\lim_{\varepsilon\to 0^+}\varepsilon^{2-p^2+(p-1)q+\sigma}|\log(\varepsilon)|^{\frac{2}{p}}=0.$$
Hence, \eqref{claimfinale} holds also if $N=p^2$.

\noindent
{\bf Case $N<p^2$}. In this direction, then  $p<\dfrac{N(p-1)}{N-p}$ so we divide the proof according to the value of $q$.

If $p<q<\dfrac{N(p-1)}{N-p}$, then we can choose $\lambda=\varepsilon^{-\sigma}$ so that by \eqref{est_norm_q}
$$\begin{aligned}J_A(t_\varepsilon u_\varepsilon)&\le c_P+ O(\varepsilon^{ \frac{\nu(N-p)}{p(p-1)}})- O(\varepsilon^{\frac{N-p}{p(p-1)}q-\sigma})
\end{aligned}$$
by choosing $\sigma$ as
$$0<\frac{N-p}{p-1}\left(\frac{q}{p}-\frac{\nu}{p}\right)<\sigma<\frac{N-p}{p-1}\cdot\frac{q}{p}$$
so that
$$0<\frac{N-p}{p(p-1)}q-\sigma<\frac{\nu(N-p)}{p(p-1)}.$$

If $q=\dfrac{N(p-1)}{N-p}$, then we can choose $\lambda=\varepsilon^{-\sigma}$ so that by \eqref{est_norm_q}
$$\begin{aligned}J_A(t_\varepsilon u_\varepsilon)&\le c_P+ O(\varepsilon^{ \frac{\nu(N-p)}{p(p-1)}})-O(\varepsilon^{\frac{N}{p}-\sigma}|\log(\varepsilon)|)
\end{aligned}$$
by choosing $\sigma$ as
$$0<\frac{N}{p}-\frac{\nu}{p}\frac{N-p}{p-1}<\sigma<\frac{N}{p}$$
so that
$$0<\frac{N}{p}-\sigma<\frac{\nu(N-p)}{p(p-1)}$$

If $\dfrac{N(p-1)}{N-p}<q<p^*$, then we can choose $\lambda=\varepsilon^{-\sigma}$ so that by \eqref{est_norm_q}
$$\begin{aligned}J_A(t_\varepsilon u_\varepsilon)&\le c_P+O(\varepsilon^{ \frac{\nu(N-p)}{p(p-1)}})- O(\varepsilon^{N-\frac{N-p}{p}q-\sigma})
\end{aligned}$$
by choosing $\sigma$ as
$$N-\frac{\nu(N-p)}{p(p-1)}-\frac{N-p}{p}q<\sigma<N-\frac{N-p}{p}q$$
so that
$$0< N-\frac{N-p}{p}q-\sigma<\frac{\nu(N-p)}{p(p-1)}.$$
Again, \eqref{claimfinale} is valid also in this latter case $N<p^2$when $\lambda$ is large enough.
Thus we have finally proved \eqref{claimfinale} under the different ranges of $\lambda$ described in $(f_3)$, which suits perfectly the proof, according to the values of $p$. Thus,  the proof is completed.

\end{proof}

Now, we are ready to prove our main result, that is the existence Theorem \ref{main}, whose statement is given
in the Introduction.

\begin{proof}[Proof of Theorem \ref{main}]
From Lemma \ref{mpt_geometry} the energy functional $J_A$ associated to problem \eqref{prob} satisfies the assumptions of Theorem \ref{mpthm}, nameley it has the mountain  pass geometry, thus there exists a Palais--Smale sequence $(u_n)_n\subset X$ of $J_A$ at level $c_A$, as pointed out
 in Remark \ref{culambda}. By Lemma \ref{c<csegnato}, then $0< c_A< c_P$, so that, according to Lemma~\ref{lem5},
there exists a nontrivial function $u\in X$ such that
$u_n \to u$ in $X$, so that $J'_A(u)\varphi=0$ for all $\varphi\in X$, that is $u$ is a solution of \eqref{prob}.
\end{proof}

\appendix

\section{Inequalities}\label{ineq}

This section concerns the proof of inequality \eqref{diaz_complex} which turned out to be useful in the final part of the proof of Lemma \ref{lem5}. Taking inspiration from \cite{lind}, let us start by recalling some well-known notions about the complex scalar product.  Take $a,b\in \mathbb C^N$ and denote $a_R,a_I,b_R,b_I\in \mathbb R^N$ their real and imaginary parts so that we can write equivalently $a=a_R+ia_I$, $b=b_R+ib_I$ or $a=(a_R,a_I)$, $b=(b_R,b_I)$.
So the scalar product in $\mathbb C^N$ can be written as
$$\begin{aligned}
\langle a,b \rangle_{\mathbb C^N}&=a\cdot \overline b=\sum_{k=1}^N a_k\overline{b_k}=
\sum_{k=1}^N (a_{kR} + ia_{kI})(b_{kR} -ib_{kI})\\
&=(a_R+ia_I)\cdot(b_R-ib_I)=a_R\cdot b_R+a_I\cdot b_I+i( a_I\cdot b_R-a_R\cdot b_I)\in\mathbb C,
\end{aligned}$$
where in the equality above $\cdot$ is the scalar product in $\mathbb R^{2N}$, by virtue of the isomorphism $\mathbb C^N\simeq R^{2N}$.
In particular, it holds $\langle a,b \rangle_{\mathbb C^N}=\overline{\langle b,a \rangle_{\mathbb C^N}}$.

For completeness, we recall an inequality from~\cite{bk} 
that will be useful in the sequel.
\begin{lemma}[{\cite[Lem.~1]{bk}}]\label{Lem.crucial}
For any non-negative real numbers $x,y,\gamma$, one has 
$$  (x+y)^\gamma \leq \alpha^\gamma x^\gamma + \beta^\gamma y^\gamma\,,$$
where $\alpha,\beta$ are any positive numbers satisfying 
$\frac{1}{\alpha} + \frac{1}{\beta} = 1$.
\end{lemma}

In what follows we report the proof of \eqref{diaz_complex} for $N=1$, but it can be repeated for $N>1$ by replacing the multiplication between real numbers with the scalar product in $\R^N$.

\begin{proof}[Proof of \eqref{diaz_complex}]
Let $a,b\in \CC$.
Observe that
$$\begin{aligned}
\lefteqn{
\langle|b|^{p-2}b-|a|^{p-2}a,
b-a\rangle_\CC}
\\
&=|b|^{p-2}\biggl(b_R^2+b_I^2-a_Rb_R-a_Ib_I\biggr)
+|a|^{p-2}\biggl(a_R^2+a_I^2-a_Rb_R-a_Ib_I\biggr)\\
& \quad+i\left[|b|^{p-2}\biggl(b_Ra_I-b_Ia_R\biggr)-|a|^{p-2}\biggl(a_Ib_R-a_Rb_I\biggr)\right].
\end{aligned}$$
Consequently,
$$\Re\langle|b|^{p-2}b-|a|^{p-2}a,
b-a\rangle_\CC=
|b|^{p-2}\biggl(b_R^2+b_I^2-a_Rb_R-a_Ib_I\biggr)
+|a|^{p-2}\biggl(a_R^2+a_I^2-a_Rb_R-a_Ib_I\biggr).$$
Since
$$|b-a|^2=(b_R-a_R)^2+(b_I-a_I)^2=
b_R^2+b_I^2+a_R^2+a_I^2-2a_Rb_R-2a_Ib_I
$$
and
$$|b|^2-|a|^2=b_R^2+b_I^2-a_R^2-a_I^2$$
so that
$$\begin{aligned}|b|^{p-2}\biggl(b_R^2+b_I^2-a_Rb_R-a_Ib_I\biggr)
&=|b|^{p-2}\biggl(\frac12|b-a|^2+\frac12\bigl(|b|^2-|a|^2\bigr)\biggr)
\end{aligned}$$
$$\begin{aligned}|a|^{p-2}\biggl(a_R^2+a_I^2-a_Rb_R-a_Ib_I\biggr)
&=|a|^{p-2}\biggl(\frac12|b-a|^2-\frac12\bigl(|b|^2-|a|^2\bigr)\biggr)
\end{aligned},$$
we reach
$$\Re\langle|b|^{p-2}b-|a|^{p-2}a,
b-a\rangle_\CC=
\frac{|b-a|^2}2\biggl(|a|^{p-2}+|b|^{p-2}\biggr)+\frac{|b|^2-|a|^2}2
\biggl(|b|^{p-2}-|a|^{p-2}\biggr).$$
Now, as in Lemma 6.1 in \cite{bi}, if $p\ge2$, by using the convexity of the function $f(y)=|y|^{p-1}$, $y\in \mathbb C$, namely that
$$\bigl(|b|^2-|a|^2\bigr)
\bigl(|b|^{p-2}-|a|^{p-2}\bigr)=\bigl(|b|+|a|\bigr)\bigl(|b|-|a|\bigr)\bigl(|b|^{p-2}-|a|^{p-2}\bigr)\ge0,$$ we get 
$$\Re\langle|b|^{p-2}b-|a|^{p-2}a,
b-a\rangle_\CC\ge 
\frac{|b-a|^2}2\bigl(|a|^{p-2}+|b|^{p-2}\bigr)\ge0.$$
The conclusion with $c=2^{2-p}$ follows from
$$ 
\begin{aligned}
|b-a|^p
=|b-a|^2|b-a|^{p-2}
&\le |b-a|^2\bigl(|a|+|b|\bigr)^{p-2}
\\
&\le 
\begin{cases}
2^{p-3}|b-a|^2\bigl(|a|^{p-2}+|b|^{p-2}\bigr)&\quad \text{if} \,\, p>3, \\
|b-a|^2\bigl(|a|^{p-2}+|b|^{p-2}\bigr)&\quad \text{if} \,\, 2\le p\le3,
\end{cases}
\end{aligned}
$$
where we have used \eqref{standard} with $\gamma=p-2$.

Now, we consider the most difficult case $1<p<2$. 
As in \cite[Sec.~10]{lind} we use the formula

$$|b|^{p-2}b- |a|^{p-2}a=\int_0^1\frac d{dt}\left[|a+t(b-a)|^{p-2}\bigl(a+t (b-a)\bigr)\right] dt$$
which yields, 
\begin{equation}\label{start}\begin{aligned}
|b|&^{p-2}b- |a|^{p-2}a=(b-a)\int_0^1|a+t(b-a)|^{p-2} dt \\
&+(p-2)\int_0^1|a+t(b-a)|^{p-4}\left[\Re\langle b-a, a+t(b-a)\rangle_{\CC}\right]\,\bigl(a+t(b-a)\bigr)dt,
\end{aligned}\end{equation}
where we have used 
$$\nabla(|v|)=\frac{\Re\langle \nabla v, v\rangle_{\CC}}{|v|}=\frac{\Re(\nabla v\cdot \overline{v})}{|v|}.$$
Thus, we obtain
$$\begin{aligned}
&\langle|b|^{p-2}b- |a|^{p-2}a, b-a\rangle_\CC=|b-a|^2\int_0^1|a+t(b-a)|^{p-2} dt \\
&+(p-2)\int_0^1|a+t(b-a)|^{p-4}\left[\Re\langle b-a, a+t(b-a)\rangle_{\CC}\right]\langle a+t(b-a),b-a\rangle_{\CC} dt
\end{aligned}$$
and consequently, we have
\begin{equation}\label{init}\begin{aligned}
    \Re&\langle|b|^{p-2}b- |a|^{p-2}a,b-a\rangle_\CC=|b-a|^2\int_0^1|a+t(b-a)|^{p-2} dt \\&+(p-2)\int_0^1|a+t(b-a)|^{p-4}\left[\Re\langle b-a, a+t(b-a)\rangle_{\CC}\right]^2dt.
\end{aligned}\end{equation}
Notice that, by using the Cauchy-Schwarz inequality, the last integral can be estimated by
\begin{equation}\label{init2}
\int_0^1|a+t(b-a)|^{p-4}\left[\Re\langle b-a, a+t(b-a)\rangle_{\CC}\right]^2dt\le |b-a|^2\int_0^1|a+t(b-a)|^{p-2}dt.
\end{equation}
On the other hand, from \eqref{init} and \eqref{init2}, we have
$$\begin{aligned}
\Re\langle|b|^{p-2}b- |a|^{p-2}a,b-a\rangle_\CC&\ge(p-1)|b-a|^2\int_0^1|a+t(b-a)|^{p-2} dt\\&\geq 
   (p-1) \, |b-a|^2 \, (|a|^2+|b|^2)^{(p-2)/2},
\end{aligned}$$
where for the second estimate we employ Lemma \ref{Lem.crucial} to get
$$|a+t(b-a)|^2 = |(1-t)a + t b|^2
  \leq (1-t)^2\alpha^2 |a|^2 + t^2 \beta^2 |b|^2
  = |a|^2 + |b|^2$$
with the choice $\gamma=2$, $\alpha = 1/(1-t)$ and $\beta=1/t$.
Consequently, by using \eqref{standard} with $\gamma=p/2$, we have
$$\begin{aligned}
  (\Re \langle |b|^{p-2} b - |a|^{p-2} a, b-a \rangle)^{p/2}
  &\geq (p-1)^{p/2} \, |b-a|^p \, (|a|^2+|b|^2)^{(p-2)p/4}  
  \\
  &\geq (p-1)^{p/2} \, |b-a|^p \, (|a|^p+|b|^p)^{(p-2)/2}  
  \,,
\end{aligned}  $$
which is \eqref{diaz_complex} with $c=(p-1)^{-p/2}$ if $1< p<2$.
\end{proof}

Note that, from \eqref{diaz_complex} for any $p\ge1$ we have the following simple fact
$$\Re\langle|b|^{p-2}b- |a|^{p-2}a,b-a\rangle_\CC>0,\qquad a\neq b.$$

We end this section by proving formula $(VI)$ in \cite{lind} in the case $p\ge 2$ and in $\mathbb C^N$, namely
$$||b|^{p-2}b- |a|^{p-2}a|\le (p-1)(|a|^{(p-2)/2}+|b|^{(p-2)/2})||b|^{(p-2)/2}b- |a|^{(p-2)/2}a|$$
for all $a,b\in\mathbb C^N$, which could be useful in a quasilinear complex scenario.
The starting point is identity \eqref{start}, which holds for any $p>1$.
The Schwarz inequality yields
$$\begin{aligned}
  \big|
  |b|^{p-2} b - |a|^{p-2} a
  \big| 
  &\leq  (p-1) \, |b-a| \int_0^1  |a+t(b-a)|^{p-2}  
  \, d t
  \\
  &= (p-1) \, |b-a| \, \left(|a|^{(p-2)/2} + |b|^{(p-2)/2}\right)
  \int_0^1  |a+t(b-a)|^{(p-2)/2} \, d t 
  \,,
\end{aligned} $$
where we employ
$$\begin{aligned}
  |a+t(b-a)|^{(p-2)/2} 
  &= |(1-t)a + t b|^{(p-2)/2} 
  \\
  &\leq (1-t)^{(p-2)/2}\alpha^{(p-2)/2} |a|^{(p-2)/2} 
  + t^{(p-2)/2} \beta^{(p-2)/2} |b|^{(p-2)/2}
  \\
  &= |a|^{(p-2)/2} + |b|^{(p-2)/2},
\end{aligned} $$
which follows from the application of Lemma \ref{Lem.crucial} with the choices $\gamma=p/2-1$, $\alpha= 1/(1-t)$ and $\beta=1/t$.
Secondly, using that $p \geq 2$ and the Schwarz inequality in~\eqref{init},
we have 
$$\begin{aligned}
  \big| |b|^{p-2} b - |a|^{p-2} a \big|
  \geq 
  |b-a| \int_0^1 |a+t(b-a)|^{p-2} \, d t
  \,;
\end{aligned} $$
and this inequality should be used 
with $p-2$ being replaced by $(p-2)/2$.

\section*{Acknowledgments}
L. B.\ and R. F.\ are members of the {\em Gruppo Nazionale per l'Analisi Ma\-te\-ma\-ti\-ca, la Probabilit\`a e le loro Applicazioni}
(GNAMPA) of the {\em Istituto Nazionale di Alta Matematica} (INdAM).
L. B. is partially supported by National Science Centre, Poland (Grant No. 2020/37/B/ST1/02742) and by the ``Maria de Maeztu'' Excellence Unit IMAG, reference CEX2020-001105-M, funded by MCIN/AEI/10.13039/501100011033/. 
R. Filippucci was partly supported by PRIN 2022 ``Advanced theoretical aspects in PDEs and their applications" Prot. 2022BCFHN.
D.K.\ was supported
by the EXPRO grant No.~20-17749X
of the Czech Science Foundation.

\end{document}